\documentclass[11pt]{amsart}

\usepackage{ amssymb,latexsym, amscd}
\usepackage{pb-diagram}
\usepackage{tikz-cd}
\usepackage[english]{babel}
\usepackage{graphicx}
\usepackage{xargs}                      % Use more than one optional parameter in a new commands
\usepackage[colorinlistoftodos,prependcaption,textsize=tiny]{todonotes}
\newcommandx{\pregunta}[2][1=]{\todo[linecolor=red,backgroundcolor=red!25,bordercolor=red,#1]{#2}}
\newcommandx{\mejorar}[2][1=]{\todo[linecolor=blue,backgroundcolor=blue!25,bordercolor=blue,#1]{#2}}

\usepackage{color}
\usepackage{amsmath}
\usepackage{enumitem}
\usepackage[all]{xy}%diagramas conmutativos
\usepackage[colorlinks = true,
            linkcolor = blue,
            urlcolor  = blue,
            citecolor = red,
            anchorcolor = blue,final]{hyperref}
\usepackage{verbatim}
\usepackage{tikz}
\usepackage{ragged2e}

\newtheorem{theorem}{Theorem}[section]

\newtheorem{lemma}[theorem]{Lemma}
\newtheorem{corollary}[theorem]{Corollary}

\newtheorem{proposition}[theorem]{Proposition}

\theoremstyle{definition}
\newtheorem{definition}[theorem]{Definition}
\newtheorem{example}[theorem]{Example}

\theoremstyle{remark}
\newtheorem{remark}[theorem]{Remark}
\newcommand{\cat}[1]{\mathcal{#1}}
\newcommand{\SV}{{\operatorname{SVec}}}
\newcommand{\z}{\mathbb{Z}}

\newcommand{\zdos}{\z/2\z}
\newcommand{\zn}[1]{\z/\displaystyle{#1}\z}
\newcommand{\FPdim}{\operatorname{FPdim}}

\newcommand{\cb}{\cat{B}}

\newcommand{\Irr}{\operatorname{Irr}}
\newcommand\Vc{\operatorname{Vec}}
\newcommand{\ot}{\otimes}

\newcommand{\id}{{\operatorname{id}}}

\newcommand{\Id}{{\operatorname{Id}}}
\newcommand{\Rep}{{\operatorname{Rep}}}
\newcommand{\Aut}[1]{{\operatorname{Aut}}(#1)}
\newcommand{\Autr}[1]{\operatorname{Aut}_{\otimes}(#1)}
\newcommand{\Autb}[1]{\operatorname{Aut}_{\otimes}^{br}(#1)}

\def\Pic#1{\text{Pic}(#1)}

\newcommand{\Ker}[1]{{\operatorname{Ker}}(#1)}

\newcommand{\mext}{\cat{M}_{ext}}

\newcommand{\picardtres}[1]{\underline{\underline{\operatorname{Pic}(#1)}}}
\newcommand{\picardos}[1]{\underline{\operatorname{Pic}(#1)}}

\def\uu#1{\underline{\underline{#1}}}

\usepackage[most]{tcolorbox}

\usepackage{soul}
\newcounter{commentcounter}

\def\cB{\mathcal{B}}
\def\cE{\mathcal{E}}

\def\cN{\mathcal{N}}
\def\cM{\mathcal{M}}

\def\cC{\mathcal{C}}

\def\cS{\mathcal{S}}
\def\cD{\mathcal{D}}

\title[Minimal modular extensions for super-tannakian categories]{Minimal modular extensions for super-tannakian categories}

\author[C. Venegas-Ram\'irez]{C\'esar F. Venegas-Ram\'irez}
\email{cf.venegas10@uniandes.edu.co}
\address{Departamento de Matem\'aticas, Universidad de los Andes, Bogot\'a, Colombia. }

\begin{document}
\maketitle
%%%%%%%%%%%%%%%%%%%%%%%%%%%%%%%%%%%%%%%%%%%%%%%% en amsart hay conflito para al hacer la lista de notas que se resuleve con los siguientes comandos %%%%%%%%%%%%%%%%%%%%%%%%%%%%%%%%%%%%%%%%%%%%%%%%%%%%%%%%%%%%%%%%%%%%%%%%%%%%%%%%%%%%%%%%%%%%%%%%%%%%%%%%%%
\makeatletter
\providecommand\@dotsep{5}
\makeatother
%\listoftodos\relax

%%%%%%%%%%%%%%%%%%%%%%%%%%%%%%%%%%%%%%%%%%%%%%%%%%%%%%%%%%%%%%%%%%%%%%%%%%%%%%%%%%%%%%%%%%%%%%%%%%%%%%%%%%%%%%%%%%%%%%%%%%%%%%%%%%%%%%%%%%%%%%%%%%%%%%%%%%%%%%%%%%%%%%%%%%%%%%%%%%%%%%%%%%%%%%%%%%%%%%%%%%%%%%%%%%%%%%%%%%%%%%%%%%%%%%%%%%%%%%%%%%%%%%%%%%%%%%%%%%%%%%%%%%%%%%%%%%%%%%%%%%%%%%%%%%%%%%%%%%%%%%%%%%%%%%%

\begin{abstract}
In this paper, we continue with the ideas presented in  \cite{galindo2017categorical}. In this opportunity, we apply the fermionic action concept to classify in cohomology terms the minimal modular extensions of a super-Tannakian category. For this goal, we study some properties of equivariantization and de-equivariantization processes and cohomology data for the fermionic case.
\end{abstract}

\section{Introduction}
Given a braided fusion category $\cB$, the construction of modular categories from $\cB$ is studied in \cite{Galois-Mueger} and \cite{Brugui?res2000} modifying the structure of $\cB$. However, it is wished to find modular categories from $\cB$ without modifying it. In other words, it is wanted to find a modular category $\cM$ with a copy of $\cB$.  In particular, modular categories with a copy of $\cB$ satisfying some minimality condition are called minimal modular extensions in \cite{Mu2}. Find the minimal modular extensions of a braided fusion category is an open problem proposed by Muger in \cite{Mu2}.

The problem of finding the minimal modular extensions for symmetric fusion category was addressed in \cite{LKW} . For a symmetric fusion category, its set of minimal modular extensions has a structure of abelian group. Moreover, for the category of representations of a finite group, called Tannakian category, was presented a complete description of all minimal modular extensions. A description of a category of representations of a super-group, called super-Tannakian categories, is left as an open problem. 

In \cite{galindo2017categorical}, it is studied the connection between the concept of fermionic action of a super-group and minimal modular extensions for a super-Tannakian category. Here, it is proposed the obstruction to the existence of minimal modular extensions for a braided fusion category. Moreover, it is proved that the homomorphism 
\[D: \mext(\Rep(\widetilde{G},z)) \to \mext(\SV) \]
defined in Equation \eqref{equation: definition D} is surjective if and only if $\widetilde{G}=G \times \zdos$.

In this paper, it is studied braided crossed fusion categories whose action is a fermionic action as well as 2-homomorphism induced of these, according to {\cite[Theorem 7.12]{ENO3}}. As a result, it is obtained a fermionic version of {\cite[Theorem 7.12]{ENO3}} in Theorem \ref{teo: fermionic crossed to 3- homomorphisms}.

It is established a correspondence between minimal modular extensions of a super-Tannakian category and braided crossed extensions with fermionic actions. This correspondence allows us to mean the minimal modular extensions as 2-homomorphisms of 2-groups; therefore, it is possible to present a description in cohomology terms. This idea is developed in Proposition \ref{prop:image of D}, Corollary \ref{cor: preimage C}, Theorem \ref{proposition: preimagen of D}, and Corollary \ref{coro:ker(D)}.

Finally, some examples are presented in Theorem \ref{theorem: zm m impar} and Example \ref{example: z4} where the results obtained here are applied .

I thank Cesar Galindo at Universidad de los Andes for time, advice, and guidance in the development of this project\footnote{This project is supported by Faculty of Science of Universidad de los Andes, Convocatoria 2018-2019 para la financiaci\'on  de proyectos de investigaci\'on y presentaci\'on de resultados en eventos ac\'ademicos categor\'ia: estudiantes de doctorado candidatos.}.

\section{Mathematical background}
\subsection{Fusion categories.} By a \emph{fusion category}, we mean a rigid monoidal category, $\mathbb{C}$-linear, semisimple,  with finite-dimensional Hom-spaces, and a finite number of isomorphism classes of simple objects that include the unit object. We denote by $\Irr(\cC)$, the set of isomorphism classes of simple objects in a fusion category $\cC$.

If  $K_0(\cC)$ denotes the Grothendieck ring, there exists a unique ring homomorphism $\FPdim : K_0(\cC)\to \mathbb{R}$ such that $\FPdim(X) > 0$ for any $X \in \Irr(\cC)$, see \cite[Proposition 3.3.6]{Book-ENO}. The  Frobenius-Perron dimension of a fusion category $\cC$ is defined as 
\[\FPdim(\cC)= \sum_{X\in \Irr(\cC)} \FPdim(X)^2.\]

\begin{example}
	Consider a finite  group $G$ and $\omega\in Z^3(G, \mathbb{C}^\times)$ a 3-cocycle with coefficients in $\mathbb{C}^\times$. $\operatorname{Vec}_G^\omega$ is the fusion category of  finite dimensional $G$-graded vector spaces, the tensor product $\otimes$ is the tensor product of $G$-graded vectors spaces, the associativity constraint is given by $a_{\delta_g,\delta_h,\delta_k}=\omega(g,h,k)id_{\delta_{ghk}}$, and the unit constraints are given by $l_{\delta_g}=\omega(e,e,g)^{-1}id_{\delta_g}$ and $r_{\delta_g}=\omega(g,e,e)id_{\delta_g}$.
	\end{example}
Given a fusion category $\cC$, the set of isomorphism classes of invertible objects is denoted by $\operatorname{Inv}(\cC)$, see \cite{Book-ENO}.

\subsection{Braided fusion categories}\label{section: braided fusion categories}
A fusion category $\cB$ is called a \emph{braided} fusion category if it is endowed with a family of natural isomorphisms 
\begin{align*}
c_{X,Y} & : X \otimes Y\to  Y \otimes X, & X,Y &\in\cC,
\end{align*}
satisfying the hexagon axioms, see \cite{js}.

If $\cB$ is a braided fusion category with braiding $c$, the reverse braided fusion category is defined  as follows: $\cB^{rev}$ is equal to $\cB$ as fusion category, but the braiding is given by $c^{rev}_{X,Y}:=c_{Y,X}^{-1}$ for $X,Y \in \cB$.

Next, we present a way to construct braided fusion categories from group cohomology data. For this, we give the abelian cohomology concept.

 Let $A$ be a finite abelian group, an \emph{abelian 3-cocycle} is pair $(\omega,c)$ such that $\omega \in Z^3(A ,\mathbb{C}^\times)$ and $c:A \times A \to \mathbb{C}^{\times}$ satisfying the following equations:
	\begin{align}\label{eq:abelian-cocycle1}
		&
		\begin{aligned}
			\frac{c(g,hk)}{c(g,h)c(g,k)}&=\frac{\omega(g,h,k)\omega(h,k,g)}{\omega(h,g,k)}\\
			\frac{c(gh,k)}{c(g,k)c(h,k)}&=\frac{\omega(g,k,h)}{\omega(g,h,k)\omega(k,g,h)}
		\end{aligned},
		& 
		\text{for all } & g,h,k \in A .
	\end{align}
	 
	We denote by $Z^3_{ab}(A , \mathbb{C}^\times)$ the abelian group of all abelian 3-cocycles $(\omega,c)$, see \cite{EM1,EM2},.
	
		An abelian 3-cocycle $(\omega,c)\in Z_{ab}^3(A ,\mathbb{C}^\times)$ is called an \emph{abelian 3-coboundary} if there is $\alpha:A ^{\times 2}\to \mathbb{C}^\times$, such that
	\begin{align*}%\label{eq:abelian-coboundary}
		&
		\begin{aligned}
			\omega(g,h,k)&=\frac{\alpha(g,h)\alpha(gh,k)}{\alpha(g,hk)\alpha(h,k)},\\
			c(g,h)&=\frac{\alpha(g,h)}{\alpha(h,g)},
		\end{aligned}
		& 
		\text{for all } & g,h,k \in A .
	\end{align*}
	 $B^3_{ab}(A ,\mathbb{C}^\times)$ denotes the subgroup of $Z_{ab}^3(A ,\mathbb{C}^\times)$ of abelian 3-coboundaries. 
	 The quotient group $H^3_{ab}(A ,\mathbb{C}^\times):=Z_{ab}^3(A ,\mathbb{C}^\times)/B^3_{ab}(A ,\mathbb{C}^\times)$ is called the \emph{third group of abelian cohomology} of $A $.

\begin{definition}
	Given $(\omega,c) \in Z^3_{ab}(A,\mathbb{C}^\times )$, we define the braided fusion category $\operatorname{Vec}_A^{(\omega,c)}$ as follows. 
	\begin{eqnarray*}
		\operatorname{Vec}_A^{(\omega,c)}= \operatorname{Vec}_A^{\omega} \text{ as fusion category.}
	\end{eqnarray*}
	The braiding of $\operatorname{Vec}_A^{(\omega,c)}$ is defined by the map $c$, and it will be denoted by the same letter.
	\begin{eqnarray*}
		c(\delta_g,\delta_h)=c(g,h)\id_{\delta_{gh}}, \text{ for each } g,h \in A.
	\end{eqnarray*} 
	The hexagon axioms are equivalent to \eqref{eq:abelian-cocycle1}.
\end{definition}

	If $(\omega,c)\in Z^3_{ab}(A,\mathbb{C}^\times)$, the map 
	\begin{align*}
		q&:A\to \mathbb{C}^\times,& q(l) &:= c(l,l), & l&\in A,
	\end{align*}
	is a quadratic form on $A$; that is, $q(l^{-1})=q(l)$ for all $l\in A$ and the symmetric map 
	\begin{align*}
		b_q(k,l):=q(kl)q(k)^{-1}q(l)^{-1},&& k,l\in A,
	\end{align*}is a bicharacter.

\begin{definition}
A pointed (braided) fusion $\cC$ is a (braided) fusion category where any simple object is invertible. The set of isormorphism classes of simple objects $A :=\operatorname{Inv}(\cC)$ is a (abelian) group with product induced by the tensor product. 
\end{definition}

It is known that a (braided) fusion category $\cC$ is (braided) equivalent to $\operatorname{Vec}_{A}^{(\omega,c)}$ for some finite (abelian) group $A$ and  (abelian) 3-cocycle.

%If $\cD$ is a full subcategory of $\cB$, the centralizer of $\cD$ respect to $\cB$ is the full subcategory $$C_{\cB}(\cD):=\{ Y \in \cb : c_{Y,X} \circ c_{X,Y}=\id_{X \otimes Y} \text{ for all } X \in \cD \}.$$

The Muger center of $\cB$ is the fusion subcategory 
\[\mathcal{Z}_2(\cB):=\{Y \in \cB: c_{Y,X} \circ c_{X,Y}=\id_{X \otimes Y }, \ \text{for all }  X \in \cB\}.\]

A braided fusion category $\cb$  is called \emph{symmetric} if  $\mathcal{Z}_2(\cb)=\cb$, i.e., if  $c_{Y,X} \circ c_{X,Y}=\id_{X \otimes Y}$ for each pair of objects $X,Y$ in $\cb$.

It is well known that symmetry fusion categories are equivalent to one of the following two examples: 

\begin{enumerate}[leftmargin=*,label=\rm{(\alph*)}]
	\item {\it Tannakian categories}. The category $\operatorname{Rep}(G)$ of finite dimensional complex representation of a finite group $G$,  with standard braiding $c_{X,Y}(x \otimes y):= y \otimes x$ for $x \in X$ and $y \in Y$.
	\item {\it Super-Tannakian categories}. \label{supergrupo}A \emph{finite super-group} is a pair $(\widetilde{G},z)$ where $\widetilde{G}$ is a finite group and $z$ is a central element of order two. An irreducible representation of $\widetilde{G}$ has one degree if $z$ acts as $-\id$, and has zero degree if $z$ acts as $\id$. 
	We denote the degree of a simple object $X \in \Rep(\widetilde{G})$  by $|X| \in \{0,1\}$.
	
	We define the  braiding $c'$ of two simple object $X$, $Y$ by
$$c'_{X,Y}(x\otimes y)= (-1)^{|X||Y|}y\otimes x.$$ 
The category $\operatorname{Rep}(\widetilde{G})$ with the braiding $c'$ is called a \emph{super-Tannakian} category, and it will be denoted by $\Rep(\widetilde{G},z)$.
\end{enumerate}

The super-Tannakian category $\operatorname{Rep}(\zdos,[1])$ is called the category of \emph{super-vector spaces} and it  will be denoted by $\SV$.

Deligne establishes that every symmetry fusion category is braided equivalent  to $\operatorname{Rep}(G)$ or  $\operatorname{Rep}(\widetilde{G},z)$ for a unique finite group $G$ or super-group $(\widetilde{G},z)$, see \cite{deligne2002categories}.

A braided fusion category $(\cB,c)$ is called \emph{non-degenerate} if $\mathcal{Z}_2(\cB) \cong \Vc$. A \emph{modular category} means a non-degenerate \emph{spherical} braided fusion categories. Non-degenerancy in this case is equivalente to the invertibility of the $S$-matrix, see \cite{DGNO}.
\begin{example}[Pointed braided fusion categories of dimension four]\label{example: categorias puntedadas dimension 4}
 Non-degenerate braided pointed fusion categories of dimension four were classified in \cite[Appendix A.3]{DGNO}, in terms of the associated metric group. Next, we present the braided structure of such categories as categories of finite dimensional $A$-graded vectors spaces $\operatorname{Vec}_A^{(\omega,c)}$.  We present abelian 3-cocycles associated to the metric groups. For this, we will identify the group of all roots of  unity in $\mathbb{C}$ with $\mathbb{Q}/\mathbb{Z}$. The description that we show below will be of great importance later.	
	\begin{enumerate}[leftmargin=*,label=\rm{(\alph*)}]
	
		\item  If $A$ has a presentation given by $A:=\{ 0,v,f, v+f: 2f=0, 2v=0 \}$  and $k\in \mathbb{Q}/\mathbb{Z}$ such that $4k=0$.
		
			We define $(\omega_k,c_k)\in Z^3_{ab}(A,\mathbb{Q}/\mathbb{Z})$ as follows: if $x=x_vv+x_ff$, $y=y_vv+y_ff$, and $z=z_vv+z_ff$ then  
	
	\begin{eqnarray*}
		\omega_k(x,y,z)&=& \begin{cases}
			0 & if \quad y_v+z_v<2\\
			2kx_v & if \quad y_v+z_v \geq 2,
		\end{cases} %\label{asoklein}
	\end{eqnarray*}
	\begin{eqnarray*}
		c_k(x,y)&=&\frac{1}{2}(x_v+x_f)y_f	+ kx_vy_v %\label{trenzaklein}.
	\end{eqnarray*}
	
	The case $k=0$ corresponds to the Drinfeld center of $\Vc_{\mathbb{Z}/2\mathbb{Z}}$ also called the Toric Code MTC. The case $k=\frac{1}{2}$ corresponds to  $(D_4, 1)$, also called three fermions MTC. The case $k=\pm \frac{1}{4}$ corresponds to two copies of Semion MTC.

\item If $A$ has a presentation given by $A:=\{ 0,v,f, v+f: 2f=0, 2v=f \}$ and $k\in \mathbb{Q}/\mathbb{Z}$ such that $4k=\frac{1}{2}$.

  We define $(\omega_k,c_k)\in Z^3_{ab}(A,\mathbb{Q}/\mathbb{Z})$ as follows: if $x=x_vv$, $y=y_vv$, and $z=z_vv$ then
	\begin{eqnarray*}
		\omega_k(x,y,z)= \begin{cases}
			0 & if \quad y_v+z_v <4 \\
			\frac{1}{2} & if \quad y_v+z_v\geq 4
		\end{cases},%\label{assz4}
	\end{eqnarray*}
	\begin{eqnarray*}
		c_k(x,y)%&=& % xyq_k(v) \nonumber \\
		&=& kx_vy_v.%\label{trenzaz4} %\text{(\cite{js})}  
	\end{eqnarray*}

	\end{enumerate}

		In all cases above, the quadratic form $q: A\to \mathbb{Q}/\mathbb{Z}$ is given by
	\begin{align*}
	q(f)=\frac{1}{2},&& q(v)=q(v+f)=k,&& q(0)=0,
	\end{align*}and the number $k=q(v)\in \{\frac{s}{8}: 0\leq s< 8 \}$ is a complete invariant.  See \cite{RSW} for more details about the classification of modular categories of dimension four.
\end{example}

\subsection{Drinfeld center of a fusion category}

An important class of non-degenerate fusion categories arises using the Drinfeld center $\mathcal{Z}(\cC)$ of a 
fusion category $(\cC,a,\mathbf{1})$, see \cite[Corolary 3.9]{DGNO}. The center construction produces a non-degenerate braided fusion category
$\mathcal{Z}(\cC)$ from any fusion category $\cC$.
Objects of $\mathcal{Z}(\cC)$ are pairs $(Z, \sigma_{-,Z})$, where $Z
\in \cC$ and $\sigma_{-,Z} : - \otimes Z \to Z \otimes -$ is a natural
isomorphism such that the diagram

{\small
 \begin{equation*}%\label{equ: centro}
\xymatrix{&  Z\ot (X\ot Y)\ar[r]^{a_{Z,X,Y}}& (Z\ot X)\ot Y &\\ 
(X\ot Y)\ot Z \ar[ru]^{\sigma_{X\ot Y,Z}}\ar[rd]_{a_{X,Y,Z}} &&& (X\ot Z)\ot Y \ar[lu]_{\sigma_{X,Z}\ot \operatorname{id}_Y}\\
& X\ot( Y\ot Z) \ar[r]_{\operatorname{id}_X\ot \sigma_{Y,Z}}&  X\ot (Z\ot Y) \ar[ru]_{a^{-1}_{X,Z,Y}}&
}
\end{equation*}
}
commutes for all $X,Y,Z \in \cC$.
The braided tensor structure is the
following:
\begin{itemize}
\item the tensor product is $ (Y, \sigma_{-,Y}) \otimes (Z, \sigma_{-,Z}) =
(Y\otimes Z, \sigma_{-,Y\otimes Z})$ where
\begin{align*}
\sigma_{X,Y \otimes Z}&:= a_{Y,Z,X}(\id_Y \otimes \sigma_{X,Z})a_{Y,X,Z}^{-1}(\sigma_{X,Y} \otimes \id_Z)a_{X,Y,Z} \textup{ for }X \in \cC.
\end{align*}
\item the braiding is the isomorphism $\sigma_{X,Y}$.
\end{itemize}

We have that 
\[\FPdim(\mathcal{Z}(\cC))=\FPdim(\cC)^2 \]
for any fusion category $\cC$ , see \cite[Proposition 9.3.4]{Book-ENO}. Morevoer, for a braided fusion category $\cb$, there is braided embedding functor

\begin{align*}
    \cB \to \mathcal{Z}(\cB),&& X\mapsto (X,c_{-X}). 
\end{align*}

\subsection{Group actions on  fusion categories}\label{categorical actions}

Let $\cC$ be a  fusion category,  $\underline{\operatorname{Aut}_\otimes(\cC)}$ denotes the monoidal category where
objects are tensor autoequivalences of $\cC$, arrows are monoidal natural isomorphisms, the tensor product is the composition of functors, and unit object $\Id_\cC$. Similarly, we define the monoidal category  $\underline{\Autb{\cB}}$ of braided autoequivalences of a braided fusion category $\cB$.

An \emph{action} of a finite group $G$ on a fusion category $\cC$  is a monoidal functor  $*:\underline{G}\to
\underline{\operatorname{Aut}_\otimes(\cC)}$  where $\underline{G}$ denotes the discrete monoidal category with objects indexed by elements of $G$ and tensor product given by the multiplication of $G$.

An action $*: \underline{G}\to \underline{\operatorname{Aut}_\otimes(\cC)}$ of $G$ over $\cC$ has the following data: 
\begin{itemize}
	\item  tensor functors $g_*: \cC\to \cC$, for each $g \in G$,
	\item  natural tensor isomorphisms $\phi(g,h): (gh)_*\to g_*\circ h_*$, for all $g,h \in G$, and
	\item  a monoidal natural isomorphism $\nu: e_* \to Id_\cC$,
\end{itemize}
which satisfy some conditions of coherence, see \cite[Section 2]{Tam-act}.

An action of a finite group $G$ on a braided fusion category $\cB$ is defined similarly. In this case the monoidal functor $*:\underline{G}\to \underline{\Autb{\cB}}$ is defined over $\underline{\Autb{\cB}}$ and the data satisfy the same coherence conditions. 

An action on a (braided) fusion category as the one described here is also called  a \emph{bosonic action}.

\begin{example}[{\cite[Section 7]{Tam-act}}]\label{Example:action-pointed fusion categories}
	
	Let $G$ and $A$  be finite groups. Given $\omega \in Z^3(A,\mathbb{C}^{\times})$, an action of $G$ on $\Vc_A^{\omega}$ is determined by a homomorphism $*:G\to \Aut{A}$ and \emph{normalized} maps
	\begin{align*}
	\mu: G\times A\times A&\to \mathbb{C}^{\times}\\
	\gamma:G\times G\times A&\to \mathbb{C}^{\times}
	\end{align*}
	such that
	\begin{align*}
	\frac{\omega(a,b,c)}{\omega(g_*(a) ,g_*(b),g_*( c))}&= \frac{\mu(g;b,c)\mu(g;a,bc)}{\mu(g;ab,c)\mu(g;a,b)},\\
	%\frac{c(a,b)}{c(\sigma_*(a),\sigma_*(b))}&= \frac{\gamma(\sigma,a,b)}{\gamma(\sigma,b,a)}\\
	\frac{\mu(g;h_*(a) ,h_*(b))\mu(h;a,b)}{\mu(gh;a,b)}&=\frac{\gamma(g,h;ab)}{\gamma(g,h ;a)\gamma(g,h;b)},\\
	\gamma(gh,k;a)\gamma(g,h;k_*(a) )&=\gamma(h,k;a)\gamma(g,hk;a),
	\end{align*}for all $a,b,c\in A$, and  $g,h,k \in G$.
	
		The action is defined as follows: for each $g\in G$, the associated monoidal functor $g_*$ is given by $g_*(\delta_a):=\delta_{g_*(a)}$, constraint $\psi(g)_{a,b}= \mu(g;a,b)\id_{\delta_{g_*(ab)}}$ and the tensor natural isomorphism is
	$$\phi(g,h)_{\delta_a}=\gamma(g,h;a)\id_{\delta_{(gh)_*(a)}},$$
	for each pair
	$g,h \in G, a\in A$.
\end{example}

We present a bosonic action of the cyclic group of order 2 on the categories $\operatorname{Vec}_A^{(\omega,c)}$ presented in 
Example \ref{example: categorias puntedadas dimension 4} where $A$ is an abelian group of order 4, and $(\omega,c)\in Z^3_{ab}(A,\mathbb{C}^\times)$.

Given that  $H^n(G,\mathbb{C}^\times)\cong H^n(G,\mathbb{Q}/\z)$, we use normalized maps $\mu$ and $\gamma$ with coefficients in $\mathbb{Q}/\z$.

\begin{example}\label{example: action pointed braided categories of foru rank}
Let $A$ be an abelian group of order 4, let $\operatorname{Vec}_A^{(\omega_k,c_k)}$ be the categories presented in Example \ref{example: categorias puntedadas dimension 4}, and let $C_2=\langle u \rangle$ be the cyclic group of order 2 generated by $u$. Then, the following data defines an action of $C_2$ over $\operatorname{Vec}_A^{(\omega_k,c_k)}$.

\begin{enumerate}[leftmargin=*,label=\rm{(\alph*)}]
    \item If $A$ has a presentation given by $A:=\{ 0,v,f, v+f: 2f=0, 2v=0 \}$, then $C_2$ has an action on $\operatorname{Vec}_A^{(\omega,c)}$ defined by 
    	\begin{align*}
	u_*(f)=f, && u_*(v)=v+f,
	\end{align*}
	where the normalized maps $\mu$ and $\gamma$ are defined by the tables: 
	\begin{align*}%\label{mu y gamma caso 1}
	\begin{tabular}{|c||c|c|c|}
	\hline
	$\mu(u;-,-)$ &  $v$ & $f$ & $f+v$ \\
	\hline 
	\hline
	$v$ & 1/2 & 1/2 & 0 \\
	\hline
	$f$ & 0& 0& 0 \\
	\hline
	$f+v$ & 1/2 & 1/2 & 0 \\
	\hline
	\end{tabular} && \begin{tabular}{|c||c|c|c|}
	\hline
	$\gamma(u,u;-)$ &  $v$ & $f$ & $f+v$ \\
	\hline 
	\hline & $\frac{1}{4}$ & 0 & $\frac{1}{4}$\\
	\hline
	\end{tabular}
	\end{align*}
	
	\item If $A$ has a presentation given by $A=\{ 0,v,f, v+f: 2f=0, 2v=f \}$, then $C_2$ has an action on $\operatorname{Vec}_A^{(\omega,c)}$ defined by
	\begin{align*}
	u_*(f)=f, && u_*(v)=v+f,
	\end{align*}
	where the normalized maps $\mu$ and $\gamma$ are defined by the tables: 
	\begin{align*}%\label{mu y gamma caso 2}
	\begin{tabular}{|c||c|c|c|}
	\hline
	$\mu(u;-,-)$ &  $v$ & $f$ & $f+v$ \\
	\hline 
	\hline
	$v$ & 0 & 1/2 & 0 \\
	\hline
	$f$ & 0& 0& 0 \\
	\hline
	$f+v$ & 0 & 1/2 & 0 \\
	\hline
	\end{tabular} && \begin{tabular}{|c||c|c|c|}
	\hline
	$\gamma(u,u;-)$ &  $v$ & $f$ & $f+v$ \\
	\hline 
	\hline & 0 & 0 & $\frac{1}{4}$\\
	\hline
	\end{tabular}
	\end{align*}
\end{enumerate}
\end{example}

\begin{definition}\label{definition: lifting }
	Let $\rho: G\to \Autr{\cC}$ be a group homomorphism where $\cC$ is a fusion category and $G$ is a finite group. A \emph{lifting} of $\rho$ is a monoidal functor $\tilde{\rho}:  \underline{G} \to \underline{\Autr{\cC}}$ such that the isomorphism class of $\tilde{\rho}(g)$ is $\rho(g)$ for each $g\in G$.
\end{definition}

If  $\rho: G\to \operatorname{Aut}_\otimes(\cC)$ is a group homomorphism, the finite group $G$  acts on $\widehat{K_0(\cC)}$. Let us fix a representative tensor autoequivalence $F_g:\cC\to \cC$ for each $g\in G$ and a  tensor natural isomorphism $\theta_{g,h}: F_g\circ F_{h}\to F_{gh}$ for each pair $g,h \in G$. Define $O_3(\rho)(g,h,l)\in \widehat{K_0(\cC)}$ by the commutativity of the diagram

\begin{equation}\label{definition 3-cocycle}
\begin{gathered} 
\xymatrixcolsep{5pc} \xymatrix{
	F_g\circ F_h\circ F_l  \ar[dd]^{F_g(\theta_{h,l})} \ar[r]^{(\theta_{g,h})_{F_l}}& F_{gh}\circ F_l \ar[d]^{\theta_{gh,l}}\\
	&F_{ghl}\ar[d]^{O_3(\rho)(g,h,l)}\\
	F_g\circ F_{hl} \ar[r]^{\theta_{g,hl}} &F_{ghl}.
}
\end{gathered}
\end{equation}

\begin{proposition}[{\cite[Theorem 5.5]{Ga1}}]\label{Proposition:Obstruction-bosonic} \label{obstruction 3}
	Let $\cC$ be a fusion category and  $\rho: G\to \operatorname{Aut}_\otimes(\cC) $ a group homomorphism. The map $O_3(\rho): G^{\times 3}\to \widehat{K_0(\cC)}$ defined by the diagram \eqref{definition 3-cocycle} is  a 3-cocycle and its cohomology class  depends on $\rho$. The map $\rho$ lifts to an action $\widetilde{\rho}:\underline{G}\to \underline{\operatorname{Aut}_\otimes(\cC)}$ if and only if $0=[O_3(\rho)]\in H_{\rho}^3(G,\widehat{K_0(\cC)} )$. If $[O_3(\rho)]=0$ the set of equivalence classes of liftings of $\rho$ is a torsor over $H^2_\rho(G,\widehat{K_0(\cC)})$.
\end{proposition}
%\qed
Proposition \ref{Proposition:Obstruction-bosonic} says us that there exists an action of $H^2_\rho(G,\widehat{K_0(\cC)})$ on the liftings of $\rho$; we denoted this action by $\triangleright$. Moreover, if $\widetilde{\rho}$ is a lifting of $\rho$, any other lifting can be obtained in the form $\mu \triangleright \widetilde{\rho}$ for $\mu \in H^2_\rho(G,\widehat{K_0(\cC)})$.

\subsubsection{Equivariantization.}
processes of equi\-variantization and de-equi\-va\-rian\-ti\-za\-tion are some of the main tools that we will use throughout this manuscript. we present a description of these processes as well as the most relevant results in this regard. Most of the results presented here appear in \cite{DGNO}.

Given an action $*:\underline{G}\to \underline{\operatorname{Aut}_\otimes(\cC)}$ of a finite group $G$ on a fusion category $\cC$ with monoidal structure given by $\phi$. The \emph{$G$-equivariantization}  of $\cC$ is the fusion category denoted by $\cC^G$ and defined as follows. An object in $\cC^G$ is a pair $(V, \tau)$, called $G$-object, where $V$ is an object of $\cC$ and $\tau$ is a family of isomorphisms $\tau_g: g_*(V) \to V$, $g \in G$, such that
\begin{equation*}%\label{deltau}
\tau_{gh}=
\tau_g g_*(\tau_h)\phi(g,h),
\end{equation*}
for all $g,h \in G$. A morphisms $\sigma: (V, \tau) \to (W, \tau')$ between $G$-objects is a morphism $\sigma: V \to W$ in $\cC$ such that
$\tau'_g\circ g_*(\sigma) = \sigma \circ \tau_g$, for all $g \in
G$. The tensor product is defined by
\begin{align*}
(V, \tau)\otimes (W, \tau'):= (V\otimes W, \tau'')
\end{align*}where $$\tau''_g= \tau_g \otimes  \tau'_g\psi(g)_{V,W}^{-1},$$and the unit
object is $(\bf{1}, \operatorname{id}_{\mathbf{1}})$. The Frobenius-Perron dimension of $\cC^G$ is $|G|\operatorname{FPdim}(\cC)$, see \cite[Proposition 4.26]{DGNO}.

%There is a canonical fully faithful monoidal functor 
%\begin{equation}\label{equation:inclusion equivariantization}
%\Rep(G)\cong\operatorname{Vec}^G \to \cC^G. 
%\end{equation} 
%This functor canonically decomposes as $\Rep(G) \to  Z(\cC^G) \to  \cC^G$. Thus $\cC^G$ is a fusion category over $\Rep(G)$, see \cite[Section 4.2.2]{DGNO}.

%\begin{proposition}\label{proposition:functor central}
%There is a canonical braided tensor functor 
%\begin{equation}
%F: \Rep(G) \to \mathcal{Z}(\cC^G)
%\end{equation}
%such that
%\begin{enumerate}
%\item The functor in (\ref{equation:inclusion equivariantization}) is isomorphic to composition of $F$ and the forgetful functor.

%\item The compositon of the functor in (\ref{equation:inclusion equivariantization}) and the forgetful functor $\mathcal{Z}(\cC^G) \to \cC$ maps $Rep(G)$ to $\operatorname{Vec}$ and has a braided structure. 
%\end{enumerate}
%\end{proposition}

\begin{theorem}[{\cite[Proposition 2.10]{ENO2}}]
	Let $\cD$ be a fusion category and let $G$ be a finite group. If there exists a braided tensor functor $\Rep(G)\to \mathcal{Z}(\cD)$ such that its composition with the forgetful functor is fully faithful, then there is a fusion category $\cC$ and an action of $G$ on $\cC$ such that $\cD \cong \cC^G$.
\end{theorem}

\subsubsection{De-equivariantization} 

In this part, we describe the opposite construction  to equivariantization called de-equivariantization.

\begin{definition}[{\cite{DGNO}}]
A \emph{central functor} from a braided fusion category $\cB$ to a fusion category $\cC$ is a braided functor $\cB \to \mathcal{Z}(\cC)$.
\end{definition}
If $\cE$ is a symmetric fusion category, $\cC$ is called a \emph{fusion category over} $\cE$ if it is endowed with a \emph{braided inclusion} $\cE \to \mathcal{Z}(\cC)$ such that its composition with the forgetful fuctor is an inclusion in $\cC$. If $\cB$ is braided, $\cB$ is a braided fusion category over $\cE$ if it is endowed with a braided inclusion $\cE \to \mathcal{Z}_2(\cB)$, see \cite{DGNO,ENO2}.

Let $\cC$ be a  fusion category and $\Rep(G)\subset \mathcal{Z}(\cC)$ be a Tannakian subcategory  which embeds into $\cC$ via the forgetful functor $\mathcal{Z}(\cC)\to \cC$. The algebra $\mathcal{O}(G)$ of functions on $G$ is a commutative algebra
in $\mathcal{Z}(\cC)$.  The category  of left $\mathcal{O}(G)$-modules in $\cC$ is a fusion category called de-equivariantization of $\cC$ by $\Rep(G)$, and it is denoted by $\cC_G$, see \cite{DGNO} for more details. It follows from  \cite[Lemma 3.11]{DGNO} that 
\[\FPdim(\cC_G)=\frac{\FPdim (\cC)}{|G|}.\]

There is a canonical fully faithful monoidal functor 

\begin{equation*}\label{equation:inclusion equivariantization}
\Rep(G)\cong\operatorname{Vec}^G \to \cC^G. 
\end{equation*} 
This functor canonically decomposes as $\Rep(G) \to  Z(\cC^G) \to  \cC^G$. Thus $\cC^G$ is a fusion category over $\Rep(G)$, see \cite[Section 4.2.2]{DGNO}.

\begin{proposition}[{\cite[Theorem 4.18, Proposition 4.19, Proposition 4.22]{DGNO}}]\label{proposition: equivalencia equivariantizacion}\label{proposition de-equivariantization} Equivariantization defines an equivalence between the 2-category of (braided) fusion categories with an action of $G$ and the 2-category of fusion categories over $\Rep(G)$. The  de-equiva\-riantization functor is inverse to the  equivariantization functor.

\end{proposition}

In short, equivariantization and de-equivariantization are mutually inverse functors.
 
\subsection{The picard group.}

Let $\cC$ be a fusion category, a left $\cC$-module category $\cM$ is a $\mathbb{C}$-linear semisimple category  with an action of $\cC$ denoted by 
\begin{align*}
 (X,M) \mapsto X*M & \text{ for }X \in \cC\text{ and } M \in \cM.
\end{align*}
The action has an associativity and unit constraint denoted by $a_{X,Y,M}: (X \otimes Y)*M \to X*(Y*M)$ and $l_M: \mathbf{1}* M\to M $ for $X,Y \in \cC$ and $M \in \cM$ that satisfy certain conditions. A right $\cC$-module category is defined in a similar way, see \cite{ENO3}

Given $\cC$ and $\cC'$ fusion categories, a $(\cC, \cC')$-bimodule category is a left $(\cC \boxtimes \cC^{'op})$-module category. 

If $\cM$ is a right $\cC$-module category and $\cN$ is a left $\cC$-module category, the tensor product of $\cM$ and $\cN$ over $\cC$ can be understood as the category of exact left $\cC$-module functors  $\cM \boxtimes_\cC \cN:= \operatorname{Fun}_{\cC,re}(\cM^{op},\cN)$, see {\cite{ENO3}}.  

In \cite[Remark 3.6]{ENO3}, it is ensured that the tensor product over $\cD$ of a $(\cC,\cD)$-bimodulo by a $(\cD,\cC')$-bimodulo has a $(\cC,\cC')$-bimodule structure. 

		\begin{definition}[\cite{ENO3}]
		A ($\cC$-$\cC'$)-bimodule category $\cM$ is invertible if there exist bimodule equivalences such that 
		\begin{eqnarray*}
		            \cM^{op} \boxtimes_\cC \cM &\cong& \cC', \text{ and }\\
		            \cM \boxtimes_{\cC'} \cM^{op} &\cong& \cC.
		\end{eqnarray*}
		
		\end{definition}

If $\cB$ is a braided fusion category, any left $\cB$-module category can be endowed with a structure of $(\cB,\cB)$-bimodule, so we can speak about invertible left $\cB$-module. 

\begin{definition}
For a braided fusion category $\cB$, the Picard 2-group is denoted by $\underline{\underline{\operatorname{Pic}(\cB)}}$ and  described as follows: Objects are invertible left $\cB$-modules, 1-morphisms are module equivalences, and 2-morphisms are isomorphisms between such equivalences.  $\underline{\underline{\operatorname{Pic}(\cB)}}$ can be truncated to a categorical group $\underline{\operatorname{Pic}(\cB)}$ if we forget the 2-morphisms and consider 1-morphisms up to isomorphism. Similarly, $\underline{\operatorname{Pic}(\cB)}$ can be truncated to the group $\operatorname{Pic}(\cB)$ called the Picard group of $\cB$.
\end{definition}

If $\cM$ is an invertible module category over $\cB$ and $\cB_{\cM}^*=\operatorname{Fun}_\cB(\cM,\cM)$, we obtain an equivalence $\cB \boxtimes \cB^{rev}\cong \mathcal{Z}(\cB_\cM^*)$. The compositions
\begin{eqnarray*}
\alpha^+&:& \cB=\cB \boxtimes \mathbf{1} \subset \cB \boxtimes \cB^{rev} \cong \mathcal{Z}(\cB_\cM^*) \to \cB_\cM^*,\\
\alpha^-&:& \cB = \mathbf{1} \boxtimes \cB^{rev} \subset \cB \boxtimes \cB^{rev} \cong \mathcal{Z}(\cB_\cM^*) \to \cB_\cM^*,
\end{eqnarray*}
are called alpha-induction functors, see \cite{ostrik1, ENO3}. In particular, for every invertible module $\cM$, the alpha-inductions are equivalences. Thus 
\[ \alpha^+=\alpha^- \circ \theta_\cM,\]
where $\theta_\cM:\cB \to \cB$ is a braided autoequivalence. For a more specific definition of the alpha-induction functors,  see \cite{davydov2013picard}. We will give a brief description.

For each $\cM \in \Pic{\cB}$ the functors  $\alpha^{\pm}$ are defined as follows:
\begin{eqnarray*}
\alpha^{\pm}_\cM: \cB \to \cB_{\cM}^*: X \to X \otimes-;
\end{eqnarray*}

The tensor structure  for $\alpha^{\pm}_\cM$ is defined by:

\begin{eqnarray}
{\Small
\xymatrix{
\alpha^+_\cM(X)(Y \otimes M)= X \otimes Y \otimes M \ar[r]^{c_{x,y}} & Y\otimes X\otimes M=Y \otimes \alpha_\cM^+(X)(M),
}}\label{eq: mult izq}\\
{\Small
\xymatrix{
\alpha^-_\cM(X)(M \otimes Y)= X \otimes Y \otimes M \ar[r]^{c^{-1}_{y,x}} & Y\otimes X\otimes M=  \alpha_\cM^-(X)(M)\otimes Y,
}\label{eq: mult der}}
\end{eqnarray}
for each $X,Y \in \cB$ and $M \in \cM$.

\begin{theorem}\label{theorem: Pic equivalente Aut}
\cite[Theorem 5.2.]{ENO3} For a non-degenerate braided fusion category $\cB$, the functor $\cM \to \theta_\cM$ is an equivalence between $\underline{\operatorname{Pic}}(\cB)$ and $\underline{\Autb{\cB}}$.
\end{theorem}

\section{Fermionic fusion categories}
All definitions and results of this section was presented in \cite{galindo2017categorical}.
\begin{definition}[\cite{galindo2017categorical}]\label{definition: spin y fermionic categories}
	Let $\cC$ be a fusion category. An object $(f,\sigma_{-,f})\in \mathcal{Z}(\cC)$ is called a \emph{fermion} if $f \otimes f \cong 1$ and  $\sigma_{f,f}= -\id_{f \otimes f}$.
	\begin{enumerate}[leftmargin=*,label=\rm{(\alph*)}]
		\item A \emph{fermionic fusion category} is a fusion category with a fermion. A fermionic fusion  category $\cC$ with fermion $(f,\sigma_{-,f})$ is denoted by the pair $(\cC,(f,\sigma_{-,f}))$.
		\item A  braided fusion category $\cb$ with braiding $c$ and a fermion of the form $(f,c_{-,f})$ is called a \emph{spin-braided fusion category}. This spin-braided fusion category will be denoted by $(\cB, f)$ because the half-braiding is determined in an obvious way by $c_{-,f}$.
	\end{enumerate}
\end{definition}

\begin{example}
The categories $\operatorname{Vec}_A^{(\omega_k,c_k)}$ presented in Example \ref{example: categorias puntedadas dimension 4} are spin-braided fusion categories with fermion $f$.
\end{example}

\begin{example}[Ising categories as spin-braided fusion categories]\label{Example:Ising}
	%The spin-braided categories of dimension four are classified in two cases, the Ising case  and the pointed case. 
	By an Ising fusion category, we mean a non-pointed fusion category of Frobenius-Perron dimension 4.
	
	Ising categories have 3 classes of 	simple objects $\mathbf{1}, f, \sigma$. The Ising fusion rules are  
	\begin{align*}
	\sigma^2=1+f,&& f^2=1,&& f\sigma=\sigma f=\sigma.    
	\end{align*}
	The associativity constraints are given
	by the $F$-matrices
	
	\begin{align*}
	F_{\sigma\sigma\sigma}^\sigma=\frac{\epsilon}{\sqrt{2}}\left( \begin{array}{cc}
	1 & 1  \\
	1 & -1 \end{array} \right),&& F_{f \sigma f}^\sigma=F_{\sigma f\sigma}^f=-1,
	\end{align*}where $\epsilon\in \{1,-1\}$. 
	
	The Ising fusion categories admit several braided structures, and in all cases $f$ is a fermion. An Ising braided fusion category  is always non-degenerate. In \cite[Appendix B]{DGNO}, it was proven that there are 8 equivalence classes of Ising braided fusion categories.
\end{example}

\begin{definition}[\cite{galindo2017categorical}]\label{fermionicfunctor}
	Let $(\cC,(f,\sigma_{-,f}))$ and $(\cC',(f',\sigma'_{-,f'}))$ be  fermionic fusion categories. A  tensor functor $(F,\tau):\cC\to \cC'$ is called a \emph{fermionic functor}  if $F(f)\cong f'$  and  the diagram 
	\begin{equation}\label{diagram:fermion functor}
	\xymatrix{
		F(V\otimes f) \ar[d]_{\tau_{V,f}} \ar[rr]^{F(\sigma_{V,f})} &&  F(f\otimes V)\ar[d]^{\tau_{f,V}} \\
		F(V)\otimes F(f)\ar[d]_{\id_{F(V)} \otimes \phi }  && F(f)\otimes F(V) \ar[d]^{\phi \otimes \id_{F(V)}} \\
		F(V) \otimes f' \ar[rr]_{\sigma'_{F(V),f'}} && f' \otimes F(V)
	} 
	\end{equation}
	commutes for each $V\in \cC$, where $\phi$ is an isomorphism between $F(f)$ and $f'$.
\end{definition}

Let $(\cC,(f, \sigma_{-,f}))$ be a fermionic fusion category. We will denote by $\underline{\Autr{\cC,f}}$ the full monoidal subcategory of $\underline{\Autr{\cC}}$ whose objects are fermionic tensor autoequivalences. The group of isomorphism classes of autoequivalences in $\underline{\Autr{\cC,f}}$ is denoted by $\Autr{\cC,f}$.

\begin{example}\label{Example: spin-braided functors}
	If $(\cB,f)$ and $(\cB',f')$ are  spin-braided fusion categories, and  $F:\cB\to \cB'$ is a braided functor such that $F(f)\cong f'$; then,  $F$ is a fermionic functor. In fact, by  definition of a braided functor, $F$ satisfies Diagram (\ref{diagram:fermion functor}), see \cite[Definition 8.1.7.]{ENO}.
	
	For spin-braided fusion categories, we will denote by $\underline{\operatorname{Aut}_\otimes^{br}(\cB,f)}$ the full monoidal subcategory of $\underline{\operatorname{Aut}_\otimes^{br}(\cB)}$ whose objects are  braided tensor autoequivalences described in Example \ref{Example: spin-braided functors}. The group of isomorphism classes of spin-braided  autoequivalences in $\underline{\Autb{\cB,f}}$ is denoted by $\Autb{\cB,f}$.
\end{example}

Next, we will present the fermionic action concept, but before we will recall some important facts.

If $(\widetilde{G},z)$ is a super-group, the exact sequence
\begin{eqnarray*}
	1 \longrightarrow \langle z\rangle \longrightarrow \widetilde{G} \longrightarrow \widetilde{G}/\langle z \rangle  \longrightarrow 1,
\end{eqnarray*}
defines and is defined by a unique element $\alpha \in H^2(\widetilde{G}/\langle z \rangle ,\zdos)$. From now on we will identify a super-group $(\widetilde{G},z)$ with the  associated pair $(G,\alpha)$ according to the convenience of the case. We set the following notation $G:=\widetilde{G}/\langle z \rangle$ and $\alpha \in H^2(G,\zdos)$.

\begin{lemma}[\cite{galindo2017categorical}]\label{lemaasignacion}
	Let  $G$ be a finite group. There is a canonical correspondence between equivalence classes of monoidal functors $\widetilde{\rho}: \underline{G} \longrightarrow \underline{ \Autb{\SV}}$ and elements of $H^2(G,\zdos)$.
\end{lemma}

Given a  functor $\widetilde{\rho}: \underline{G} \longrightarrow \underline{ \Autb{\SV}}$, the corresponding element in $H^2(G,\zdos)$ will be denoted by $\theta_{\widetilde{\rho}}$.

\begin{definition}[\cite{galindo2017categorical}]\label{accion}
	Let $(\cC,(f,\sigma_{-,f}))$ be a fermionic fusion category and $(G,\alpha)$ a super-group. A \emph{fermionic action} of $(G,\alpha)$ on $(\cC,(f,\sigma_{-,f}))$ is a monoidal functor
	\begin{eqnarray*}
		\widetilde{\rho}: \underline{G} \longrightarrow \underline{\operatorname{Aut}_\otimes(\cC,f)} ,
	\end{eqnarray*}
	such that the restriction functor $\widetilde{\rho}: \underline{G} \longrightarrow \underline{\Autr{\langle f \rangle}}$ satisfies  $\theta_{\widetilde{\rho}} \cong \alpha$ in $H^2(G,\zdos)$, see Lemma \ref{lemaasignacion}.

	A fermionic action of a finite super-group $(G,\alpha)$ on a spin-braided fusion category $(\cB,f)$ is defined in a similar way. In this case,  the monoidal functor $\widetilde{\rho}:\underline{G}\to \underline{\Autb{\cB,f}}$ is defined over $\underline{\Autb{\cB,f}}$ and the data satisfies the same condition. 
\end{definition}

Two of the most important results presented in \cite{galindo2017categorical} are the following: 

\begin{theorem}[\cite{galindo2017categorical}]\label{theorem: equivalencia dequivariantizacion- equivariantizacion}
	Let $(\widetilde{G},z)$ be a finite super-group. Then the equivariantization and de-equivariantization processes define a biequivalence of 2-categories between fermionic fusion categories with a fermion action of  $(\widetilde{G},z)$ and fusion categories over $\Rep(\widetilde{G},z)$.
\end{theorem}

\begin{corollary}[\cite{galindo2017categorical}]\label{coro: equivalencia dequivariantizacion- equivariantizacion braided case}
	Let $(\widetilde{G},z)$ be a finite super-group. Then equivariantization and de-equivariantization processes define a biequivalence of 2-categories between spin-braided fusion categories with fermionic action of  $(\widetilde{G},z)$ compatible with the braiding, and braided fusion categories  $\cD$ over $\Rep(\widetilde{G},z)$ such that  $\Rep(G)\subseteq \mathcal{Z}_2(\cD)$.
\end{corollary}

Given a finite super-group $(G,\alpha)$, a fermionic fusion category $(\cC, (f,\sigma_{-,f}))$, and a  group homomorphism $\rho: G \to \Autr{\cC,f}$. An $\alpha$-lifting of $\rho$ is a fermionic action  $\widetilde{\rho}: \underline{G} \to \underline{\Autr{\cC,f}}$ of $(\widetilde{G},\alpha)$ over $(\cC, (f,\sigma_{-,f}))$ such that the isomorphism class of $\widetilde{\rho}(g)$ is $\rho(g)$ for each $g \in G$.

\section{Obstruction to fermionic actions}\label{section: Obstruction to fermionic actions}

If   $\rho: G \to \Autr{\cC,f}$ is a group homomorphism, the question that we answer in this part is when there exists a fermionic action that realizes $\rho$, i.e., we want to know when there exists a $\alpha$-lifting of $\rho$.

The fermionic obstruction was defined in  \cite{galindo2017categorical} to determine when a group homomorphism $\rho$ is an $\alpha$-lifting.

The existence of an $\alpha$-lifting implies the existence of a lifting for $\rho$ in the sense of Definition \ref{definition: lifting }. Thus, we have that the obstruction $O_3$ in Theorem \ref{Proposition:Obstruction-bosonic} vanishes. 

Let us recall the fermionic obstruction. We consider the $G$-module homomorphism $r:\widehat{K_0(\cC)}\to\widehat{K_0(\langle f \rangle)}\cong \zdos$ defined by restriction.  When $r$ is non-trivial, the exact sequence

\begin{eqnarray*}%\label{sequencek0}
	\xymatrix{
		1 \ar[r] & \Ker{r} \ar@{^{(}->}[r]^i & \widehat{K_0(\cC)} \ar@{->>}[r]^r & \zdos \ar[r] & 1
	}
\end{eqnarray*}
induces a long exact sequence 

{%\Small
	\begin{eqnarray*}
	\xymatrix{
		\,\,\,\, \ar[r] & H^2(G, Ker(r)) \ar[r]^{i_*} & H^2(G,\widehat{K_0(\cC)}) \ar[r]^{r_*}& H^2(G, \zdos) \ar[r]^-{d_2}&  \\
		\ar[r] & H^3(G, Ker(r)) \ar[r]^{i_*} & H^3(G,\widehat{K_0(\cC)}) \ar[r]^{r_*}& H^3(G, \zdos) \ar[r]^-{d_3}& \ldots
	} %\label{sucesionexacta}
	\end{eqnarray*}}

This long exact sequence is used in the following definition.

\begin{definition}\label{O3alt}
	Let $\rho: G\to \Autr{\cC,f}$ be a group homomorphism and $\widetilde{\rho}: \underline{G} \longrightarrow \underline{\Autr{\cC,f}}$ a lifting of $\rho$. For each $\alpha \in H^2(G,\zdos)$, we define 
	\begin{equation*}%\label{def:obstructionalt}
	O_3({\rho}, \alpha):= \begin{cases}
	\theta_{\widetilde{\rho}}/\alpha \in H^2(G,\zdos) & if \quad r \quad\text{ is trivial, }\\
	d_2(\theta_{\widetilde{\rho}}/\alpha)\in  H^3(G,\operatorname{Ker}(r)) & if \quad r \quad \text{ is  non-trivial,}
	\end{cases}
	\end{equation*}
	where $r:\widehat{K_0(\cC)}\to\widehat{K_0(\langle f \rangle)}$ is the restriction map defined above.
\end{definition}

In Theorem \ref{theorem:bstruction fermionicaction}, we establish the independence of the obstruction on the choice of the lifting  $\widetilde{\rho}$, the existence of an $\alpha$-lifting in terms of a cohomological value, and a correspondence of liftings with a certain subgroup in $H^2(G, \widehat{K_0(\cC)})$.

\begin{theorem}[\cite{galindo2017categorical}]\label{theorem:bstruction fermionicaction}
	
	Let $(G,\alpha)$ be a finite super-group and $\rho: G \to \Autr{\cC,f}$ a group homomorhism with $\widetilde{\rho}: \underline{G} \to \underline{\Autr{\cC,f}}$ a lifting of $\rho$. Then
	\begin{enumerate}[leftmargin=*,label=\rm{(\alph*)}]
		\item the element $O_3({\rho}, \alpha)$ does not depend on the  lifting.
		\item The homomorphism $\rho$ has lifting to a fermionic action of $(G,\alpha)$ if and only if $O_3({\rho},\alpha)=0$.
		\item The set  of  equivalence  classes  of  $\alpha$-liftings of $\rho$ is  a  torsor  over 
		\[\operatorname{Ker}\Big( r_*: H^2(G,\widehat{K_0(\cC)}) \to  H^2(G, \zdos) \Big ).\]
	\end{enumerate}
\end{theorem}

\section[Fermionic actions on spin-braided fusion categories of dimension four]{ Fermionic actions on non-degenerate spin-braided fusion categories of dimension four}\label{accionsv}

Up to equivalence, all braided fusion categories with dimension four was presented in Example \ref{example: categorias puntedadas dimension 4}, as pointed braided fusion categories with fusion rules given by an abelian group of order four, and in Example \ref{Example:Ising}, as Ising categories. 

Next, we present some facts about fermionic actions over spin-braided fusion categories of dimension four. For more details can be consulted \cite{galindo2017categorical}.
\begin{proposition}\label{proposition: fermionicactionsonIsing}
Only trivial super-groups  act fermionically on a  spin-braided Ising category. 
\end{proposition}

\begin{theorem}\label{therorem: extensions of A with order 4}
Let $(G,\alpha)$ be a finite super-group and  $\Vc_{A}^{(\omega_k,c_k)}$ a pointed  spin-modular category of dimension four. Then
\begin{enumerate}[leftmargin=*,label=\rm{(\alph*)}]
    \item $\Autb{\Vc_{A}^{(\omega_k,c_k)},f}\cong \mathbb{Z}/2\mathbb{Z}$.
    \item A group homomorphism $G\to\mathbb{Z}/2\mathbb{Z}\cong  \Autb{\Vc_{A}^{(\omega_k,c_k)},f}$  is always realized by a bosonic action.
    \item A group homomorphism $G\to\mathbb{Z}/2\mathbb{Z}\cong  \Autb{\Vc_{A}^{(\omega_k,c_k)},f}$  is realized by a fermionic action of a super-groups $(G,\alpha)$  associated to $\rho$ if and only if $d_2(\alpha)=0$. 
    \item If $d_2(\alpha)=0$ then the equivalence classes of fermionic actions of $(G,\alpha)$ associated to $\rho$ is a torsor over \[\operatorname{Ker}\Big( r_*: H^2(G,A) \to  H^2(G, \zdos) \Big ).\]
\end{enumerate}
Here $d_2: H^2(G,\zdos)\to H^3(G,\zdos)$ is the connecting homomorphism associated to the $G$-module exact sequence $0\to \zdos \to A\overset{r}{\to} \zdos \to 0$.
\end{theorem}

\section{\texorpdfstring{Braided $(\widetilde{G},z)$-crossed extensions}%
                               {Braided ({G},z)-crossed extensions}} \label{chapter: BRAIDED (G,z)-CROSSED EXTENSIONS}

In this section, we finish establishing some properties of fermionic actions on fermionic fusion categories that arise under the processes of equivariantization and de-equivariantization. These properties refer exclusively to the classification of a particular type of extensions of fermionic fusion categories, which we call braided $(\widetilde{G},z)$-crossed extensions.

\subsection{\texorpdfstring{Braided $G$-crossed fusion categories}%
{Braided G-crossed fusion categories}} \label{section: Braided G-crossed fusion categories}

\begin{definition}
Let $G$ be a finite group. A $G$-grading for a fusion category $\cC$ is a decomposition
\[ \cC= \bigoplus_{g \in G} \cC_g \]
into a direct sum of full abelian subcategories such that the tensor product defines a functor from $\cC_g\times \cC_h$ into $\cC_h$ for all $g,h \in G$. We assume that the grading is faithful, i.e., $\cC_g \neq 0$ for all $g \in G$.
\end{definition}

\begin{definition}
A $G$-extension of a fusion category $\cD$ is a $G$-graded fusion category $\cC$ such that $\cC_e$ is equivalent to $\cD$.
\end{definition}
\begin{definition}[\cite{turaev5291homotopy}]

A fusion category $\cC$ is called a braided $G$-crossed fusion category if it is equipped with the following data:

\begin{enumerate}[leftmargin=*,label=\rm{(\alph*)}]
\item a grading $\cC= \bigoplus_{g \in G}\cC_g$,
\item an action $\underline{G}\to \underline{\Autr{\cC}}$ of $G$ on $\cC$ such that $g_*(\cC_h)\subset \cC_{ghg^{-1}}$, and
\item a $G$-\emph{braiding}, that is, natural isomorphisms 
\begin{eqnarray*}
c_{X,Y}: X \otimes Y \to g_*(Y)\otimes X, \qquad  X \in \cC_g, g \in G, \text{ and }Y\in \cC.
\end{eqnarray*}
\end{enumerate}

If $\phi_{g,h}: (gh)_* \to g_*h_*$ is the monoidal structure of the functor $g \to g_*$, and $\mu_g$ is the tensor structure of $g_*$, we need to hold some compatibility conditions.
\end{definition}

Note that the trivial component $\cC_e$ is itself a braided fusion category with an action of $G$ by braided autoequivalences of $\cC_e$.

\begin{theorem}[\cite{DGNO}] \label{theorem: equivalencia briaded G-crossed extensions}
	The equi\-va\-riantization and de-equi\-va\-rian\-tization constructions define a bijection betweeen equivalence classes of braided $G$-crossed fusion categories and equivalence classes of braided fusion categories containing $\Rep(G)$ as a symmetric fusion subcategory.
\end{theorem}

Theorem \ref{theorem: equivalencia briaded G-crossed extensions} tells us that the de-equivariantizacion by $G$ of a braided fusion category containing $\Rep(G)$ is a braided G-crossed fusion category. The following theorem shows how we can find the trivial component of de-equivariantization and the commutativity relation between taking centralizers and taking de-equivariantization. These results can be found in \cite[Theorem 3.8]{turaev2010homotopy} and \cite[Proposition 4.30]{DGNO}.

\begin{proposition}\label{proposition: componente trivial de-equivariantizacion}
Let $\cB$ be a braided fusion category containing $\Rep(G)$ as a full subcategory, and $\cD$ be the de-equivariantization of $\cB$ by $G$.
\begin{enumerate}[leftmargin=*,label=\rm{(\alph*)}]
    \item\label{formula componente trivial} The trivial component of the braided $G$-crossed fusion category $\cD$ is $C_\cB(\Rep(G))_G$. In particular, if $\mathcal{Z}_2(C_\cB(\Rep(G)))=\Rep(G)$ the braided fusion category $\cD_e$ is modular.
    
    \item \label{isomorphismo lattice}Equivariantization and de-equivariantization define an isomorphism between the lattice of fusion subcategories of $\cB$ containing $\Rep(G)$ and the lattice of $G$-stable fusion subcategories of $\cD$, i.e., if $\widetilde{\cD}$ is a $G$-stable subcategory of $\cD$ then $\widetilde{\cD}^G=\widetilde{\cB}$ contains $\Rep(G)$, and if $\widetilde{\cB}$ is a subcategory of $\cB$ containing  $\Rep(G)$ then $\widetilde{\cB}_G=\widetilde{\cD}$ is a $G$-stable subcategory of $\cD$. 
    \item \label{centralizador y  equivariantizacion}Suppose $\cB$ is a braided fusion category over $\Rep(G)$. The isomorphism of \ref{isomorphismo lattice} commutes with taking centralizer, i.e., $C_\cB(\widetilde{\cD}^G)=(C_\cD(\widetilde{\cD}))^G$ and $C_\cD(\widetilde{B}_G)=(C_\cB(\widetilde{B}))_G$.
\end{enumerate}
\end{proposition}

\begin{theorem}[{\cite[Theorem 1.3.]{ENO3}}]\label{theorem: extensiones brauer-picard} Graded extensions of a fusion category $\cC$ by a finite group $G$ are parametrized by triples $(c,M,\alpha)$, where $c: G \to \operatorname{BrPic}(\cC)$ is a group homomorphism, $M$ belongs to a certain torsor $T_c^2$ over $H^2(G, \operatorname{Inv}(\mathcal{Z}(\cC))) $ (where $G$ acts on $\operatorname{Inv}(\mathcal{Z}(\cC))$ via $c$), and $\alpha$ belongs to a certain torsor $T^3_{c,M}$ over $H^3(G,\mathbb{C}^\times)$. Here the data $c$, $M$ must satisfy the conditions that certain obstruction $O_3(c)\in H^3(G,\operatorname{Inv}(\mathcal{Z}(\cC)))$ and $O_4(c,M) \in H^4(G,\mathbb{C}^\times) $ vanish.
%\comargen{revisar si define que es una extension graduada}
\end{theorem}

Theorem \ref{theorem: extensiones brauer-picard} establishes the basis for the classification of all extensions of fusion categories, but we are interested in  particular  braided $G$-crossed extensions that we presente below . For those extensions Theorem \ref{theorem: extensiones picard} is more relevant.

\begin{theorem}[{\cite[Theorem 7.12]{ENO3}}]\label{theorem: extensiones picard} Let $\cB$ be a braided fusion category. Equivalence classes of braided $G$-crossed extensions of $\cB$ (with faithful $G$-grading) are in bijection  with morphisms of categorical 2-groups $\underline{\underline{G}} \to \underline{\underline{\operatorname{Pic}(\cB)}}$.

\end{theorem}

\subsection{\texorpdfstring{$H^4$-Obstruction}%
{H4-Obstruction}} \label{section: h4 obstruction}
 The obstruction that we will study next, called $H^4$-obstruction, measures when a bosonic action on a non-degenerate fusion category  $\widetilde{\rho}: \underline{G} \to \underline{\Autb{\cB}}$ can be lifted to a 2-group homomorphism $\widetilde{\widetilde{\rho}}: \underline{\underline{G}} \to \underline{\underline{\operatorname{Pic}(\cB)}}$. According to Theorem \ref{theorem: Pic equivalente Aut}, we are identifying ${\underline{\Autb{\cB}}}$ with $\underline{\operatorname{Pic}(\cB)}$.

We want to construct a 2-homomorphism from $\underline{\underline{G}}$ to $\underline{\underline{\operatorname{Pic}(\cB)}}$ for a non-degenerate braided fusion category. We can characterize this 2-homomorphisms in cohomology terms. Firtly, Computing the $O_3$-obstruction to know if a group homomorphism $G \to \Autb{\cB}$ admits a lifting $\widetilde{\rho}$, and secondly, computing  the $H^4$-obstruction to determine if  it has a lifting to a 2-homomrphism. Since 2-homomorphisms classify braided $G$-crossed extensions of $\cB$, the process gives us an algorithm to find these braided $G$-crossed extensions. That is one of our goals, specially extensions with specific properties that we will discuss later.

Suppose a non-degenerate braided fusion category $\cb$, a bosonic action $\widetilde{\rho}$ is determined by data $\widetilde{\rho}:=(g_*,\psi^g,\varphi_{g,h}): \underline{G} \to \underline{\Autb{\cb}}$, and the set of equivalence classes of bosonic actions is a torsor over $H^2_\rho(G,\operatorname{Inv}(\cb))$. If we suppose that $\widetilde{\rho}$ admits a lifiting $\widetilde{\widetilde{\rho}}:\underline{\underline{G}} \to \uu{\Pic{\cB}}$, a 2-cocycle $\mu \in Z^2_\rho(G,\operatorname{Inv}(\cb))$  has a bosonic action denoted by $\mu \triangleright \widetilde{\rho}$.

The \emph{$H^4$-obstruction} of the pair $(\widetilde{\rho},\mu)$ is defined as a 4-cocycle $O_4(\widetilde{\rho},\mu) \in H^4(G,\mathbb{C}^\times)$ described by the formula

\begin{eqnarray}\label{equation: h4-obstruction}
            O_4(\widetilde{\rho},\mu) &=& c_{\mu_{g_1,g_2},(g_1g_2)_*(\mu_{g_3,g_4})}\\
            & & a_{(g_1g_2)_*(\mu_{g_3,g_4}),\mu_{g_1,g_2},\mu_{g_1g_2,g_3g_4}} \nonumber \\
            & & a^{-1}_{(g_1g_2)_*(\mu_{g_3,g_4}),(g_1)_*(\mu_{g_2,g_3g_4}),\mu_{g_1,g_2g_3g_4}} \nonumber \\
            & & a_{(g_1)_*(\mu_{g_2,g_3}),(g_1)_*(\mu_{g_2g_3,g_4}), \mu_{g_1g_2g_3,g_4}} \nonumber \\
            & & a^{-1}_{(g_1)_*(\mu_{g_2,g_3}),\mu_{g_1,g_2g_3},\mu_{g_1g_2g_3,g_4}} \nonumber \\
            & & a_{\mu_{g_1,g_2}, \mu_{g_1g_2,g_3},\mu_{g_1g_2g_3,g_4}} \nonumber \\
            & & a^{-1}_{\mu_{g_1,g_2},(g_1g_2)_*(\mu_{g_3,g_4}),\mu_{g_1g_2,g_3g_4}} \nonumber \\
            & & \varphi_{g_1,g_2}(\mu_{g_3,g_4}) \nonumber \\
            & & (\psi^{g_1})^{-1}((g_2)_*(\mu_{g_3,g_4}),\mu_{g_2,g_3g_4}) \nonumber \\
            & & \psi^{g_1}(\mu_{g_2,g_3},\mu_{g_2g_3,g_4}) ,\nonumber 
\end{eqnarray}
where $a$ is the associative constraint of the category $\cB$.
\begin{proposition}[{\cite[Proposition 9]{SCJZ}}] \label{proposition: O4-obstruction}
If $\cb$ is a non-degenerate braided fusion category, the homomorphism of categorical groups $(\mu \triangleright \widetilde{\rho}): \underline{G} \longrightarrow \underline{\operatorname{Pic}(\cb)}$ can be lifted if and only if $O_4(\widetilde{\rho},\mu)$ defined by (\ref{equation: h4-obstruction}) is trivial.
\end{proposition}

More details about the $H^4$-obstruction used in this paper can be consulted in \cite{ENO3} and \cite{SCJZ}.

\subsection{\texorpdfstring{Braided $(\widetilde{G},z)$-crossed extensions}%
                               {Braided ({G},z)-crossed extensions}} \label{section: braided super-extensions}
{

Of all braided $G$-crossed extensions of a braided fusion category $\cB$, we are interested in studying those extensions with a fermionic action of $(\widetilde{G},z)$ over  the trivial component. The idea of this part is to classify this ``fermionic'' extensions in terms of 2-homomorphisms as happens in Theorem \ref{theorem: extensiones picard}.

	\begin{definition}
		A \emph{braided $(\widetilde{G},z)$-crossed fusion category} $(\cD,f)$ is a braided $G$-crossed fusion category $\cD$ where $(\cD_e,f)$ is a spin-braided fusion category in the sense of Definition \ref{definition: spin y fermionic categories}, and the action that corresponds to the structure of braided $G$-crossed category is a fermionic action of $(\widetilde{G},z)$ on $(\cD_e,f)$.
	\end{definition}
	
	In this case, a functor between braided $(\widetilde{G},z)$-crossed fusion categories is a functor between braided $G$-crossed fusion categories which is also a fermionic functor. 
	
	The next corollary  is the fermionic version of Theorem \ref{theorem: equivalencia briaded G-crossed extensions}. This one gives a bijection between braided fusion categories over $\Rep(\widetilde{G},z)$ and braided $(\widetilde{G},z)$-crossed fusion categories.
	
	\begin{corollary}\label{col: equivalence fermionic crossed}
		Let $(\widetilde{G},z)$ be a finite super-group. Equivariantization and de-equivariantization define a bijection between braided $(\widetilde{G},z)$-crossed fusion categories $(\cD,f)$, up to equivalence, and braided fusion categories $\cC$ over $\Rep(\widetilde{G},z)$, up to equivalence.
	\end{corollary}
	
	\begin{proof}
		There is a bijection between braided $G$-crossed fusion categories $\cD$ and braided fusion categories over $\Rep(G)$ by Theorem \ref{theorem: extensiones picard}. 
		
		If $\cC$ is a braided fusion category over $\Rep(\widetilde{G},z)$, it is a braided fusion category over  $\Rep(G)$, so $\cD:=\cC_G$ is a braided $G$-crossed fusion category, and $\cD$ is a fermionic fusion category  by Theorem \ref{theorem: equivalencia dequivariantizacion- equivariantizacion}.
		
		Finally, as $\Rep(\widetilde{G},z) \subseteq \cat{Z}_2(\Rep(G))$ then $\SV \subseteq \cD_e$, and the fermion in $\cD$ belongs to its trivial component.
		
		Conversely, if $(\cD,f)$ is a braided $(\widetilde{G},z)$-crossed fusion category, the equivariantization $\cC=\cD^G$ is a braided fusion category such that $\Rep(\widetilde{G},z)=\SV^G\subseteq \cC$.
	\end{proof}
	
	\begin{definition}
	A \emph{braided $(\widetilde{G},z)$-crossed extension} of a spin-braided fusion category $(\cB,f)$ is a braided $(\widetilde{G},z)$-crossed fusion category $(\cD,f)$ whose trivial component $(\cD_e,f)$ is equivalent to $(\cB,f)$.
	\end{definition}
	
	{
		\begin{definition}\label{definition: pic(b,f)}
			If $(\cb,f)$ is a spin-braided fusion category, we consider $\picardos{\cb,f} \subseteq \picardos{\cb}$ as the full subcategory with objects $\cM\in \Pic{\cB}$ such that $\theta_\cM(f)=f$.
		\end{definition}
		
		The condition $\theta_\cM(f)=f$ for an invertible module category $\cM$ is equivalent to say that the module functors $ - \otimes f $ and $f \otimes - $ are isomorphic autoequivalences of $\cM$. In fact, this is a consequence of the definition of the functors $\alpha^{\pm}$ in (\ref{eq: mult izq}), (\ref{eq: mult der}) and $\theta$. More details can be found in \cite{davydov2013picard}.
	}
	
	\begin{proposition}\label{prop:equivalence fermionic functors}
		Consider a non-degenerate spin-braided fusion category $(\cb,f)$. There is an equivalence between $\underline{\Autb{\cB,f}}$ and $\underline{\operatorname{Pic}(\cb,f)}$.
	\end{proposition}
	
	{
	\begin{proof}
	This proposition is a direct consequence of Definition \ref{definition: pic(b,f)}. In fact, $\theta: \underline{\Pic{\cB}} \to \underline{\Autb{\cB}}$ is an equivalence of categories. In particular, for each $\cM \in \underline{\operatorname{Pic}(\cB,f)}$, we have $\theta_\cM(f)=f$ by definition, so $\theta_{\cM} \in \underline{\Autb{\cB,f}}$.
	\end{proof}
	}

	{
	Theorem \ref{teo: fermionic crossed to 3- homomorphisms} classifies braided $(\widetilde{G},z)$-crossed extensions of spin-braided fusion categories in terms of homomorphisms of 2-groups from $\underline{\underline{G}}$ to $\underline{\underline{\operatorname{Pic}(\cB,f)}}$. This theorem is the key result that we use to find minimal modular extensions for super-Tannakian categories, as we will see later.

		\begin{theorem}\label{teo: fermionic crossed to 3- homomorphisms}
			Let $(\cb,f)$ be a spin-braided fusion category. Equivalence classes of braided $(\widetilde{G},z)$-crossed categories $\cC$ having a faithful $G$-grading with trivial component $\cb$ are in bijection with homomorphism of categorical $2$-groups $\widetilde{\widetilde{\rho}}:\underline{\underline{G}} \to \picardtres{\cb,f}$, such that $\widetilde{\rho}$ is a fermionic action of $(\widetilde{G},z)$ on $(\cB,f)$.
		\end{theorem}
		
		\begin{proof}
			According to Theorem \ref{theorem: extensiones picard}, a 2-group homomorphism $\widetilde{\widetilde{\rho}}: \underline{\underline{G}} \to \picardtres{\cb,f}$ corresponds to a braided $G$-crossed extension $\cD$ of $\cb$. The action of $G$ on $\cD_e=\cB$ is given by $\widetilde{\rho}$, so $\cD$ is a braided $(\widetilde{G},z)$-crossed extension of $\cB$.
			
			Conversely, if $\cD$ is a braided $(\widetilde{G},z)$-crossed extension of $\cB$, it corresponds to a $2$-group homomorphism $\widetilde{\widetilde{\rho}}: \underline{\underline{G}} \to \picardtres{\cB}$, whose action induced is fermionic since the action on $\cB$ is fermionic too. Now, the $G$-action for $g \in G$ on $\cb$ is given by the equivalence of $\operatorname{Fun}_{\cb}(\cD_g,\cD_g)$ with the  left and right  multiplication by elements of $\cB$, see \cite{ENO3}. In particular, if $\cD_g \in \operatorname{Pic}(\cb,f)$ then $g_*$ is a fermionic functor, i.e., $g_* \in \Autb{\cb,f}$. %As $r_*( \theta_{\widetilde{\rho}})= \alpha$ the action thus defined is a fermionic action of  $(\widetilde{G},z)$ on $(\cb,f)$.
		\end{proof}
	}
}

% Chapter Template

\section{Minimal modular extensions}\label{chapter: MINIMAL NON-DEGENERATE EXTENSIONS} % Main chapter title
Muger defines minimal modular extensions of a braided fusion category $\cB$  in \cite{Galois-Mueger}. For a Tannakian category, a complete description of its minimal modular extensions was presented in \cite{LKW}. Nevertheless,  for the super-Tannakian case, such description is yet an open problem.  We use Corollary \ref{col: equivalence fermionic crossed} and Theorem \ref{teo: fermionic crossed to 3- homomorphisms} to give an approach to the solution.

\subsection{Minimal modular extensions}\label{section: minimal modular extensions}

\begin{definition}\label{definition: minimal modular extensions }
Let $\cB$ be a braided fusion category. A minimal modular extension of $\cat{B}$ is a pair $(\cat{M},i)$, where  $\cat{M}$ is a modular fusion category such that  $i:\cat{B}\to \cat{M}$ is a braided full embedding and 
\[ \operatorname{FPdim}(\cM)=\operatorname{FPdim}(\cB)\operatorname{FPdim}(\mathcal{Z}_2(\cB)). \]

Two minimal modular extensions $(\cM,i)$ and $(\cM',i')$ are equi\-valent if there exists a braided equivalence $F: \cM \longrightarrow \cM'$ such that $F\circ i\cong i'$.
\end{definition}

\begin{example}[Modular extensions of $\SV$]
 For the symmetric super-Tannakian category $\SV$, there are 16 modular extensions (up to equivalence). They can be classified in two classes, the first one is given by 8 Ising braided modular categories parametrized by $\zeta$, such that $\zeta^8=-1$. A brief description  of them was given in Example \ref{Example:Ising}. The second one is given by 8 pointed modular categories $\operatorname{Vec}_A^{(\omega_k,c_k)}$ where $A$ is an abelian group of order four, and $(\omega_k,c_k)$ is an abelian 3-cocycle. A description of this type of categories was presented in Example \ref{example: action pointed braided categories of foru rank}. More information about this example can be found in
\cite{Kitaev20062, DGNO}.     
\end{example}

In \cite{LKW} the set of equivalence classes of minimal modular  extensions of a unitary braided fusion category $\cB$ is denoted by $\mathcal{M}_{ext}(\cb)$. In particular, if $\cb$ is a  symmetric fusion category, $\mathcal{M}_{ext}(\cb)$ is an abelian group with unit object $\mathcal{Z}(\cB)$.

Let $\cE$ be a symmetric fusion category; the set of minimal modular extensions of $\cE$ is non-empty since $\mathcal{Z}(\cE)$ is always a minimal modular extension. In this case, any minimal modular extension $\cM$ of $\cE$ can be thought as a module category over $\cE$ with action induced by tensor product of $\cM$. Then, the binary operation on $\mathcal{M}_{ext}(\cE)$ is defined using the tensor product of module categories over $\cE$. This operation is well defined according to \cite[Lemma 4.11]{LKW}. Moreover, the associativity of this operation is proved in \cite[Proposition 4.12]{LKW} in a more general case.

The existence of the neutral element was proved in \cite[Lemma 4.18]{LKW}. There is shown that for any symmetric category $\cE$, the Drinfel center $\mathcal{Z}(\cE)$ is the neutral object in $\mext(\cE)$.

\begin{example}[Modular extensions of Tannakian fusion categories]\label{bosonicsymmetry}
	For a symmetric Tannakian category $\Rep (G)$,  the group of modular extensions (up to equivalence) is isomorphic to the abelian group $H^3(G,\mathbb{C}^\times)$. For each $\omega \in H^3(G,\mathbb{C}^\times )$, $\mathcal{Z}(\operatorname{Vec}_G^\omega)$ is a modular extension of $\Rep(G)$, see \cite{LKW}.
\end{example}

\subsection{Obstruction theory to existence of minimal modular extensions}\label{subsection: Obstruction theory to existence of minimal non-degenerate extensions}

Following \cite{Galois-Mueger, modularizacion}, we say that a braided fusion category is \emph{modularizable} if $\mathcal{Z}_2(\cB)$ is Tannakian. Note that if $\Rep(G)=\mathcal{Z}_2(\cB)$ then $\cB_G$ is a modular fusion category; therefore, the term modularizable makes sense from this point of view.
\begin{definition}[\cite{galindo2017categorical}]\label{definition: H4-anomaly}
Let $\cB$ be a modularizable braided fusion category with $\mathcal{Z}_2(\cB)=\Rep(G)$. The de-equivariantization $\cB_G$ has associated a monoidal functor
\[
\xymatrix{ 	& \underline{G} \ar[r]^-{\widetilde{\rho}} &  \underline{\operatorname{Aut}_\otimes^{br}(\cB_{G})} \ar[r]^{\Phi} &  \underline{\Pic{\cB_G}}\\
},
\]
where $\widetilde{\rho}$ is the canonical action of $G$ on the modular  fusion category $\cb_G$.
We define the \emph{$H^4$-anomaly} of $\cB$ as the $H^4$-obstruction of $\widetilde{\rho}$ in $H^4(G,\mathbb{C}^\times)$.
\end{definition}

According to Proposition \ref{proposition: O4-obstruction}, trivial anomaly  is equivalent to saying that $\cB_G$ has a $G$-extension with induced action on it given by $\widetilde{\rho}$. In fact, if the $O_4$-obstruction is trivial, there exists a lifing of $\widetilde{\rho}$ to a 2-homomorphism $\widetilde{\widetilde{\rho}}: \underline{\underline{G}} \to \underline{\underline{\operatorname{Pic}(\cB_G)}}$. Thus,  according to Theorem \ref{theorem: extensiones picard}, there exists a braided $G$-crossed extension of $\cB_G$.
\begin{definition}[\cite{ENO2}]
A braided fusion category $\cb$ is called \emph{slightly degenerate} if $\mathcal{Z}_2(\cb)$ is braided equivalent to $\SV$.
\end{definition}

 In \cite{galindo2017categorical}, we have mentioned some relations between slightly degenerate fusion categories and categories over  $\Rep(\widetilde{G},z)$ under fermionic actions . Specifically, if the Muger center of $\cB$ is a super-Tannakian category $\Rep(\widetilde{G},z)$, the category $\cB_G$ is a slightly degenerate fusion category where $G=\widetilde{G}/\langle z \rangle$.

 Theorem \ref{theorem:teorema de-equivariantizacion y extensiones} tells us that the study of minimal modular extensions for a non-modularizable fusion category can be reduced to the study of certain associate slightly degenerate fusion category. 

If $\cb$ is non-modularizable, that is $\mathcal{Z}_2(\cb) \cong \Rep(\widetilde{G},z)$,  the maximal central tannakian subcategory of $\cb$ is braided equivalent to $\Rep(G)$ with  $G \cong \widetilde{G}/\langle z \rangle$.

\begin{theorem}[\cite{galindo2017categorical}]\label{theorem:teorema de-equivariantizacion y extensiones}
Let $\cB$ be a braided fusion category with non-trivial maximal central Tannakian subcategory $\Rep(G) \subseteq \mathcal{Z}_2(\cb)$. 

\begin{enumerate}[leftmargin=*,label=\rm{(\alph*)}]
    \item \label{theorem 4.10 part 1}If $\cB$ is modularizable, $\cB$ admits a minimal modular extension if and only if the $H^4$-anomaly of $\cB$ vanishes. 
    \item\label{theorem 4.10 part 2} If $\cB$ is non-modularizable with $\mathcal{Z}_2(\cb)=\Rep(\widetilde{G},z)$,  $\cB$ admits a minimal modular extension if and only if the following conditions hold:
    \begin{enumerate}[leftmargin=*,label=\rm{(\roman*)}]
\item\label{2a} the slightly degenerate braided fusion category $\cB_G$ has a minimal modular extension $\cS$, 

\item\label{2b} there exists a fermionic action of $(\widetilde{G},z)$ on $\cS$ such that $\cB_G$ is $G$-stable, and the restriction to $\cb_G$ coincides with the canonical action of $G$ on $\cb_G$,

\item\label{2c}the anomaly of $\cS^G$ vanishes.
    \end{enumerate}
\end{enumerate}
\end{theorem}

\begin{remark}\label{comentario: sobre teorema e interpretacion a traves de braided super-crossed extension}
\begin{enumerate}[leftmargin=*,label=\rm{(\alph*)}]
\item Another way of expressing part \ref{theorem 4.10 part 2} of Theorem \ref{theorem:teorema de-equivariantizacion y extensiones} is the following: a non-modularizable braided fusion category $\cB$ has a minimal modular extension if and only if the slightly degenerate fusion category $\cB_G$ has a minimal modular extension $\cS$ which in turn has a braided $(\widetilde{G},z)$-crossed extension.
\item Theorem \ref{theorem:teorema de-equivariantizacion y extensiones} was used in \cite{galindo2017categorical} to show examples of braided fusion categories without minimal modular extensions. 
\end{enumerate}
\end{remark}

The de-equivariantization defines a well-defined map 
\begin{eqnarray*}
            D: &\mathcal{M}_{ext}(\cb)& \to \mathcal{M}_{ext}(\cb_G) \label{equation: mapa D forma general}\\
            &\cM& \to (\cM_G)_e.\nonumber
\end{eqnarray*}
 For any minimal modular category $\cM$ of $\cB$ the map $D$ sends $\cM$ to the trivial component of the de-equivariantization  $\cM_G$, see \cite{LKW} and \cite{galindo2017categorical}.

In particular, for a super-Tannakian category $\Rep(\widetilde{G},z)$ with maximal central Tannakian subcategory $\Rep(G)$, 
\begin{eqnarray}\label{equation: definicion D}
            D: \mathcal{M}_{ext}(\Rep(\widetilde{G},z)) \to \mathcal{M}_{ext}(\SV)  \label{equation: definition D}
\end{eqnarray}
is a group homomorphism.

	\begin{corollary}[\cite{galindo2017categorical}]\label{corol:Ostrik}
Let $\Rep{(\widetilde{G},z)}$ be a finite super-group. The map \[D:\mathcal{M}_{ext}(\Rep{(\widetilde{G},z)})\to \mathcal{M}_{ext}(\SV)\] is surjective if and only if $(\widetilde{G},z)$ is a trivial super-group.
\end{corollary}

\begin{theorem}[\cite{galindo2017categorical}]\label{extensiontriviasupergroup}
Let $(\cb,f)$ be a spin-braided fusion category of dimension four, and $G$ be a finite group with a group homomorphism $\xi: G \longrightarrow \Autb{\cb,f}$.  Then $\xi $ can be extended to a bosonic action  and  to a 2-homomorphism $\widetilde{\widetilde{\xi}}: \underline{\underline{G}} \longrightarrow \picardtres{\cb}$.
\end{theorem}
{ 

\subsection{The group \texorpdfstring{$\cat{M}_{ext}(\operatorname{Rep}(\widetilde{G},z))$.} .}\label{section: mext group}

We use the homomorphism $D$, Corollary \ref{col: equivalence fermionic crossed}, and Theorem \ref{teo: fermionic crossed to 3- homomorphisms} to describe in cohomology terms the minimal modular extensions of a super-Tannakian category $\Rep(\widetilde{G},z)$.

In this subsection, we denote by $\cb$  a pointed braided fusion category of dimension four presented in Example \ref{example: action pointed braided categories of foru rank}, by $\cat{I}$ an Ising category (see  Example \ref{Example:Ising}), and  by
\begin{eqnarray*}
            D: \mext(\Rep(\widetilde{G},z)) \to \mext (\SV)
\end{eqnarray*}
the group homomorphism defined in  (\ref{equation: definition D}). We fix the following notation, $(\widetilde{G},z)$ is a finite super-group, $G$ is the quotient group $\widetilde{G}/\langle z  \rangle$, and the pair $(G,\alpha)$ is identified with the super-group $(\widetilde{G},z)$ where $\alpha \in Z^2(G,\zdos)$ is the unique 2-cocycle (up to cohomology) associated to the super-group $(\widetilde{G},z)$. 

By definition of $D$, for a minimal modular extension $\cM$ of $\Rep(\widetilde{G},z)$, $D(\cM)$ is a minimal modular extension of $\SV$. In fact, we have that $\Rep(\widetilde{G},z)_G \cong \SV$, and $D(\cM)=(\cM_G)_e $ is modular with $\operatorname{FPdim}((\cM_G)_e)=4$.

On the other hand, note that every minimal modular extension $\cC$ of $\SV$ has a natural structure of spin-braided fusion category where the natural fermion corresponds to the inclusion of  $\SV$ in $\cC$.

\begin{proposition}\label{prop:image of D}
	Let $\cC$ be a minimal modular extension of $\SV$. The spin-braided fusion category $\cC$ is in the image of $D$ if and only if $\cC$ has a braided $(\widetilde{G},z)$-crossed extension. 
\end{proposition}

\begin{proof}
	Consider a minimal modular extension $\cC$ of $\SV$ in the image of $D$, then there is $\cM \in \mext(\Rep(\widetilde{G},z))$ such that $D(\cM)=\cC$. By Theorem \ref{theorem:teorema de-equivariantizacion y extensiones}, $\cM_G$ is a braided $(\widetilde{G},z)$-crossed extension of $\cC$.
	
	If $\cC \in \mext(\SV)$ has a braided $(\widetilde{G},z)$-crossed extension $\cat{L}$, the modular fusion category $\cat{L}^G$ is a category on $\Rep(\widetilde{G},z)$ according to Corollary \ref{col: equivalence fermionic crossed}. Moreover, by Theorem \ref{theorem:teorema de-equivariantizacion y extensiones},  $\cat{L}^G$ is a minimal modular extension of $\Rep(\widetilde{G},z)$.
	\end{proof}

\begin{corollary}\label{cor: preimage C}
Let $\cC$ be a minimal modular extension of $\SV$. The pre-image of  $\cC$ with respect to $D$ is in correspondence with 2-homomorphisms $\widetilde{\widetilde{\rho}}:\underline{\underline{G}} \to \underline{\underline{\operatorname{Pic}(\cC,f)}}$ such that the truncation $\widetilde{\rho}$ is a fermionic action of $(\widetilde{G}, z)$.
\end{corollary}
\begin{proof}
    Using Proposition \ref{prop:image of D}, for each $\cC \in \operatorname{Im}(D)$, there is an $\cM$ in $\mext(\Rep(\widetilde{G},z))$ such that $\cM_G$ is a braided $(\widetilde{G},z)$-crossed extension of $\cC$. By Theorem \ref{teo: fermionic crossed to 3- homomorphisms},  braided $(\widetilde{G},z)$-crossed extensions of $\cC$ correspond to 2-group homomorphisms $\widetilde{\widetilde{\rho}}:\underline{\underline{G}} \to \underline{\underline{\operatorname{Pic}(\cC,f)}}$  such that $\widetilde{\rho}$ is a fermionic action of the super-group $(\widetilde{G},z)$.
	\end{proof}
	
We know that braided $G$-crossed extensions of a non-degenerate fusion category $\cC$ are in correspondence with 2-homomorphisms $\widetilde{ \widetilde{\rho}}: \underline{\underline{G}} \to \underline{\underline{\operatorname{Pic}(\cC)}}$. If $\widetilde{\rho}: \underline{G} \to \underline{\Autb{\cC}}$ is the truncation of $\widetilde{\widetilde{\rho}}$ then liftings of $\widetilde{\rho}$ are a torsor on $H^3(G,\mathbb{C}^\times)$.  In a similar way, if $\rho$ is the truncation of $\widetilde{\rho}$, liftings of $\rho$ are a torsor on $H^2_\rho(G, \widehat{K_0(\cC)})$. Therefore, any 2-homomorphism associated by truncation to $\rho$ can be parametrized by an element in $H^2_\rho(G, \widehat{K_0(\cC)}) \times H^3(G,\mathbb{C}^\times)$. Conversely, if we start with a group homorphism $\rho$ and consider the obstruction theory in order to obtain a 2-homomorphism, we can conclude that any lifting to a 2-homomorphism can be parametrized by pairs $(\mu,\varphi) \in H^2_\rho(G, \widehat{K_0(\cC)}) \times H^3(G,\mathbb{C}^\times)$ such that  the obstructions $O_3(\rho)$ and $O_4(\rho,\mu)$ vanish. Thus, every 2-homomorphism $\widetilde{ \widetilde{\rho}}: \underline{\underline{G}} \to \underline{\underline{\operatorname{Pic}(\cC)}}$ can be parametrized by triples $(\rho,\mu,\varphi)$,  where $\rho:G \to \Autb{\cC,f}$ is a group homomorphism, $\mu$ belongs to a certain torsor over $H^2_\rho(G,\widehat{K_0(\cC)})$, and $\varphi$ belongs to a certain torsor over $H^3(G,\mathbb{C}^\times)$ such that $O_3(\rho)$ and  $O_4(\rho,\mu)$ vanish. 

In Theorem \ref{proposition: preimagen of D} below, we characterize the image of the group homomorphism $D$ in terms of group homomorphisms and group cohomology.

\begin{theorem}\label{proposition: preimagen of D}
	Consider a minimal modular extension $\cC$ of $\SV$. The pre-image of $\cC$ under $D$ is parametrized by triples $(\rho, \mu, \varphi)$, where $\rho:G \to \Autb{\cC,f}$ is a group homomorphism, $\mu$ belongs to a certain torsor over $\operatorname{Ker}(r_*:H^2_\rho(G,\widehat{K_0(\cC)} \to H^2_\rho(G,\widehat{K_0(\SV)} )$, and $\varphi$ belongs to a certain torsor over $H^3(G,\mathbb{C}^\times)$. The data $\mu$ and $\varphi$ must satisfy the conditions that obstruction $O_3(\rho, \alpha)$ and  $O_4(\rho,\mu)$ vanish.
\end{theorem}

\begin{proof}

Let $\cC$ be a minimal modular extension of $\SV$. By Corollary \ref{cor: preimage C}, the pre-image of $\cC$ is in correspondence with  2-homomorphisms $\widetilde{\widetilde{\rho}}: \underline{\underline{G}} \to \underline{\underline{\operatorname{Pic}(\cC,f)}}$ such that $\widetilde{\rho}$ is a fermionic action. These 2-homomorphisms are in correspondence with data $(\rho, \mu, \varphi)$ where $\rho: G \to \Autb{\cC,f}$, $\mu$ belongs to certain torsor over $\operatorname{Ker}(r_*: H^2(G,\widehat{K_0(\cC)} \to H^2(G,\zdos)))$ meaning the fermionic condition, and $\varphi$ belong to certain torsor over $H^3(G, \mathbb{C}^\times)$. Moreover, the data satisfies that the obstructions $O_3(\rho,\alpha)$ and $O_4(\rho,\mu)$ vanish.

Conversely, We know that data $(\rho, \mu, \varphi)$ with obstructions $O_3(\rho)$ and $O_4(\rho,\mu)$  vanish parametrize  2-homomorphisms $\widetilde{\widetilde{\rho}}: \underline{\underline{G}} \to \underline{\underline{\operatorname{Pic}(\cC)}}$. The conditions $O_3(\rho,\alpha)=0$ and $\mu \in \operatorname{Ker}(r_*: H^2(G,\widehat{K_0(\cC)})\to H^2(G,\zdos)))$ is equivalente to say that the 2-homomrphisms $\widetilde{\widetilde{\rho}}$ have values in $\underline{\underline{\operatorname{Pic}(\cC,f)}}$ with truncation to a ferminic action. 
	\end{proof}

	\begin{remark}
The Proposition \ref{proposition: preimagen of D} tells us that the image of  $D$ correspond to spin-braided categories $(\cC,f)$ in $\mext(\SV)$ with at least one fermionic action that can be extended to 2-homomorphisms $\widetilde{\widetilde{\rho}}: \underline{\underline{G}}\to \underline{\underline{\operatorname{Pic}(\cC,f)}}$. 
	\end{remark}

\begin{corollary}\label{coro:ker(D)}
	The kernel of  $D$ is parametrized by triples $(\rho, \mu, \varphi)$ where $\rho:G \to \Autb{\operatorname{Vec}_{\zdos \times \zdos}^{(\omega_0,c_0)},f}$ is a group homomorphism, $\mu$ belongs to a certain torsor over $\operatorname{Ker}(r_*:H^2_\rho(G,\widehat{K_0(\cC)}) \to H^2_\rho(G,\widehat{K_0(\SV)} )$, and $\varphi$ belongs to a certain torsor over $H^3(G,\mathbb{C}^\times)$ such that $O_4(\rho,\mu)$ vanishes.
\end{corollary}

\begin{proof}
    The result follows from Theorem \ref{proposition: preimagen of D}, given that the trivial element of $\mext(\SV)$ is $\mathcal{Z}(\SV) \cong \operatorname{Vec}_{\zdos \times \zdos}^{(\omega_0,c_0)}$ as we explain in Section  \ref{section: minimal modular extensions}. 
	% By group structure explained in Section \ref{section: minimal modular extensions}, the trivial element of $\mext(\SV)$ is $\mathcal{Z}(\SV)$, that correspond to $\operatorname{Vec}_{\zdos \times \zdos}^{(\omega_0,c_0)}$ in Example \ref{example: action pointed braided categories of foru rank}. Then, the result is a consequence of \ref{proposition: preimagen of D}.%kernel of $D$ correspond to minimal modular extensions $\cat{M}$ such that $\cat{M}_G$ is a braided $(\widetilde{G},z)$-crossed extension of $\operatorname{Vec}_{\zdos \times \zdos}^{(\omega_0,c_0)}$
	\end{proof}

Next, we use Proposition \ref{proposition: preimagen of D} and Corollary \ref{coro:ker(D)} to determine the order of $\mext(\Rep(\zn{m} \times \zdos, (0,1)))$ for  $m$ an odd number.

\begin{theorem}[Minimal modular extensions of $(\zn{m}, \alpha \equiv 0 )$, $m$ odd]\label{theorem: zm m impar}
	Consider the trivial super-group $\zn{m} \times \zdos$ where $m$ is an odd number. The group $\mext(\Rep({\zn{m}\times \zdos,([0],[1])}))$ has order $16m$.	
\end{theorem}
\begin{proof}
We have that 
	\begin{enumerate}[leftmargin=*,label=\rm{(\alph*)}]
		\item By Corollary \ref{corol:Ostrik} the group homomorphism $D$ is surjective. 
		\item 	Given that $m$ is an odd number, the unique group homomorphism $\zn{m} \to \zdos \times \zdos$ is the trivial homomorphism. \label{example: zm m impar part 2}
		\item $H^2(\zn{m}, \zdos \times \zdos)\cong 0$, \label{example: zm m impar part 3}
		\item $H^3(\zn{m},\mathbb{C}^\times ) = \zn{m} $. \label{example: zm m impar part 4}
	\end{enumerate}	

	 Using items (\ref{example: zm m impar part 2}), (\ref{example: zm m impar part 3}),  and (\ref{example: zm m impar part 4}),   the Corollary \ref{coro:ker(D)} implies that  kernel of $D$ is correspondence with triples $(\rho, \mu, \varphi)$ where $\rho:\zn{m} \to \Autb{\operatorname{Vec}_{\zdos \times \zdos}^{\omega_0,c_0},f}$ is the trivial homomorphism, $\mu \in H^2_\rho(G,\zdos \times \zdos)=0$, and $\varphi$ belongs to a certain torsor over $H^3(\zn{m},\mathbb{C}^\times)$. As there are $m$ of that  triples, the order of $\mext(\Rep({\zn{m}\times \zdos,([0],[1])})$ is $16m$. 
\end{proof}

\begin{example}[Minimal modular extensions of $\zn{6}$]
 The group $\mext(\Rep(\zn{6},[3]))$ has order $48$. We consider $\zn{6}\cong \zn{3} \times \zn{2}$ and  apply  Theorem \ref{theorem: zm m impar}. Then the group of minimal modular extensions of $\Rep(\zn{6},[3])$ has order 48.
 
 This result agrees with \cite[Table XX]{lan2016classification}.
\end{example}

\begin{example}[Minimal modular extensions of $\zn{4}$]\label{example: z4}
 we can deduce that there are exactly 32 minimal modular extensions of $\Rep({\zn{4},[2]})$. This information agrees with the result presented by Ostrik. In this case, we prove that $\ker(D)$ has order 4 and that the image of $D$ consist of the pointed modular extensions of $\SV$.
 
 \begin{enumerate}[leftmargin=*,label=\rm{(\alph*)}]
 	\item The $H^4$-obstruction that we need to consider is $H^4(\zdos, \mathbb{C}^\times)=0$, so each action $\underline{\zdos} \to {\underline{\Autb{\cC,f}}}$ has a lifiting to a 2-homomorphism. 
 	
 	\item $\zn{4}$ is not a trivial super-group, so no Ising category can  be in the image of $D$.
 	
 	\item The kernel of $D$ is parametrized by triples $(\rho,\mu,\varphi)$ where $\rho: \zdos \to \Autb{\operatorname{Vec}_{\zdos \times \zdos}^{(\omega_0, c_0)},f}$ is the trivial homomorphism, $\mu \in H^2(\zdos, \zdos \times \zdos)$ such that $r_*(\mu)$ is non-trivial, and $\varphi \in H^3(\zdos, \mathbb{C}^\times) \cong \zdos$. Then, there are $4$ such triples, which implies that $\operatorname{Ker}(D)$ has order $4$. A similar analysis shows that every pointed fusion category with fusion rules given by $\zdos \times \zdos$ is in the image of $D$.
 	
 	\item Moreover, every pointed fusion category with fusion rules given by $\zn{4}$ is in the image of $D$. In fact, the triple $(\rho, \mu, \varphi)$ satisfies the conditions in  Proposition \ref{proposition: preimagen of D}; where $\rho: \zdos \to \Autb{\operatorname{Vec}^{(\omega_k,c_k)}_{\mathbb{Z}/4\mathbb{Z}} ,f}$ is the trivial homomorphism, $\mu$ is the unique non-trivial object in $H^2(\zdos,\zn{4})$, and $\varphi \in H^3(\zdos,\mathbb{C}^\times)$.
 \end{enumerate}
\end{example}
  
  \begin{proposition}\label{proposisition: image has 4 elements}
      Let $D:\mext(\Rep(\widetilde{G},z)) \to \mext(\SV)$ be  the group homomorphism defined above. We have that $D$ is non-trivial and the image of $D$ has at least 4 elements. Specifically, the pointed fusion categories  with fusion rules given by $\zdos \times \zdos$ is always  in the image of $D$.
  \end{proposition}
  
  \begin{proof}
  In general, consider $\cB$ one of the pointed braided fusion categories in Example \ref{example: action pointed braided categories of foru rank}. Remember that $\cB$ is non-degenerate and the restriction map $r: \widehat{K_0(\cB)}\cong \operatorname{Inv}(\cB) \to  \widehat{K_0(\SV)}$ can be consider as $r(X)(f)=c_{f,X}\circ c_{X,f}$, for each $X \in \operatorname{Inv}(\cB)$. According to the braided structure of each $\cB$ the group homomorphism $r$ is the same for each $\cB$ with fusion rules $\zdos \times \zdos$. This implies that the group homomorphisms and the group cohomology in Theorem \ref{proposition: preimagen of D} are the same.
\end{proof}

  \begin{proposition}
  	Let $\cb$ be a slightly degenerate pointed braided fusion category, then $\cb$ has a minimal modular extension. 
  \end{proposition}
  
  \begin{proof}
  	By \cite[Proposition 2.6]{etingof2011weakly} $\cb \cong \SV \boxtimes \cb_0$, where $\cb_0$ is a non-degenerate pointed fusion category. Take $\cM$ a minimal modular extension of $\SV$ and consider de modular category $\cM \boxtimes \cb_0$.
  	\end{proof}

  	  	%\begin{proposition}
  	  	%	Let $\cb$ be a slightly degenerate fusion category with minimal extension $\widetilde{\cb}$. $\cb$ is nilpotent %(respectively, weakly group theoretical, solvable) if and only if $\widetilde{\cb}$ is nilpotent (respectively, weakly group theoretical, solvable).
 	  	%\end{proposition}
   	}

%%%%%%%%%%%%%%%%%%%%%%%%%%%%%%%%%%%%%%%%%%%%%%%%bibliografia%%%%%%%%%%%%%%%%%%%%%%%%%%%%%%%%%%%%%%%%%%%%%%%%%%%%%%%%%%%%%%%%%%%%%%%%%%%
%%%%%%%%%%%%%%%%%%%%%%%%%%%%%%%%%%%%%%%%%%%%%%%%%%%%%%%%%%%%%%%%%%%%%%%%%%%%%%%%%%%%%%%%%%%%%%%%%%%%%%%%%%%%%%%%%%%%%%%%%%%%%%%%%%%%%%%

\bibliographystyle{alpha}
\bibliography{biblio}
\end{document}